\documentclass[12pt,a4paper]{amsart}
\usepackage{amsmath,amsthm,amssymb}
\usepackage{graphicx}
\usepackage{subfigure}

\DeclareMathOperator{\RE}{Re} \DeclareMathOperator{\IM}{Im}
\theoremstyle{plain}
\newtheorem{theorem}{Theorem}[section]
\newtheorem{lemma}[theorem]{Lemma}
\newtheorem{corollary}[theorem]{Corollary}

\theoremstyle{remark}
\newtheorem{remark}[theorem]{Remark}
\newtheorem{example}[theorem]{Example}

\begin{document}

\title[convolution of harmonic mappings]{Univalence and convexity in one direction of the convolution of harmonic mappings}

\author{Sumit Nagpal}
\address{Department of Mathematics, University of Delhi,
Delhi--110 007, India} 

\author{V. Ravichandran}
\address{Department of Mathematics, University of Delhi,
Delhi--110 007, India} 
\email{vravi68@gmail.com}

\begin{abstract}
Let $\mathcal{H}$ denote the class of all complex-valued harmonic functions $f$ in the open unit disk normalized by $f(0)=0=f_{z}(0)-1=f_{\bar{z}}(0)$, and let $\mathcal{A}$ be the subclass of $\mathcal{H}$ consisting of normalized analytic functions. For $\phi \in \mathcal{A}$, let $\mathcal{W}_{H}^{-}(\phi):=\{f=h+\bar{g} \in \mathcal{H}:h-g=\phi\}$ and $\mathcal{W}_{H}^{+}(\phi):=\{f=h+\bar{g} \in \mathcal{H}:h+g=\phi\}$ be subfamilies of $\mathcal{H}$. In this paper, we shall determine the conditions under which the harmonic convolution $f_1*f_2$ is univalent and convex in one direction if $f_1 \in \mathcal{W}_{H}^{-}(z)$ and $f_2 \in \mathcal{W}_{H}^{-}(\phi)$. A similar analysis is carried out if $f_1 \in \mathcal{W}_{H}^{-}(z)$ and $f_2 \in \mathcal{W}_{H}^{+}(\phi)$. Examples of univalent harmonic mappings constructed by way of convolution are also presented.\medskip
\end{abstract}
\keywords{univalent harmonic mappings; sense-preserving; convex in one direction; convolution}

\subjclass{Primary 31A05; Secondary 30C45}

\maketitle

\section{Introduction}
Let $\mathcal{H}$ denote the class of all complex-valued harmonic functions $f$ in the open unit disk $\mathbb{D}=\{z \in \mathbb{C}:|z|<1\}$ normalized by $f(0)=0=f_{z}(0)-1=f_{\bar{z}}(0)$. Such functions can be written in the form $f=h+\bar{g}$, where
\begin{equation}\label{eq1.1}
h(z)=z+\sum_{n=2}^{\infty}a_{n}z^{n}\quad\mbox{and}\quad g(z)=\sum_{n=2}^{\infty}b_{n}z^n
\end{equation}
are analytic in $\mathbb{D}$. The quantity $w_f=g'/h'$ is called the dilatation of a harmonic mapping $f$. A function $f\in \mathcal{H}$ is sense-preserving in $\mathbb{D}$ if the Jacobian $J_{f}=|h'|^2-|g'|^2$ is positive or equivalently $|g'|<|h'| $ in $\mathbb{D}$. The dilatation $w_f$ of a sense-preserving harmonic mapping $f$ is analytic and satisfies $|w_f|< 1$ in $\mathbb{D}$. Let $\mathcal{S}_{H}^{0}$ be the subclass of $\mathcal{H}$ consisting of univalent and sense-preserving functions. The classical family $\mathcal{S}$ of normalized analytic univalent functions is a subclass of $\mathcal{S}_{H}^{0}$. Finally, let $\mathcal{K}$ and $\mathcal{K}_{H}^{0}$ be the subclasses of $\mathcal{S}$ and $\mathcal{S}_{H}^{0}$ respectively, mapping $\mathbb{D}$ onto a convex domain.

Clunie and Sheil-Small \cite{cluniesheilsmall} introduced  the method of ``shear construction'' to produce a harmonic mapping with a specified dilatation onto a domain convex in one direction by shearing a given conformal mapping along parallel lines. A domain $\Omega \subset \mathbb{C}$ is convex in the direction of real (resp.\  imaginary) axis if its intersection with each horizontal (resp.\  vertical) line is connected. A function $f \in \mathcal{H}$ is convex in the direction of real (resp.\  imaginary) axis if it maps $\mathbb{D}$ onto a domain convex in the direction of real (resp.\  imaginary) axis. The shear construction is contained in the following lemma.

\begin{lemma}\cite{cluniesheilsmall}\label{lem1}
A sense-preserving harmonic function $f=h+\bar{g}$ in $\mathbb{D}$ is a univalent mapping of $\mathbb{D}$ convex in the direction of real (resp.\  imaginary) axis if and only if $h-g$ (resp.\  $h+g$) is a conformal univalent mapping of $\mathbb{D}$ convex in the direction of the real (resp.\  imaginary) axis.
\end{lemma}

For analytic functions $\phi(z)=z+\sum_{n=2}^{\infty}a_{n}z^{n}$ and $\psi(z)=z+\sum_{n=2}^{\infty}A_{n}z^{n}$, their convolution (or Hadamard product) is defined as $(\phi*\psi)(z)=z+\sum_{n=2}^{\infty}a_{n}A_{n}z^{n}$, $z\in \mathbb{D}$. In the harmonic case, with $f=h+\bar{g}$ and $F=H+\bar{G}$, their harmonic convolution is defined as $f*F=h*H+\overline{g*G}$. The right half-plane mapping $l(z)=z/(1-z)$ act as convolution identity for analytic functions, while the function $e(z)=z/(1-z)+\overline{z}^2/(1-\overline{z})\in \mathcal{H}$ is an identity under the harmonic convolutions. Note that $e \not\in \mathcal{S}_{H}^{0}$. Harmonic convolutions are investigated in \cite{cluniesheilsmall,dorff1,dorff2,goodloe,sumit1,sumit5,ruscheweyh}.

Let $\mathcal{A}$ be the subclass of $\mathcal{H}$ consisting of normalized analytic functions. For $\phi \in \mathcal{A}$, consider the following subfamilies of $\mathcal{H}$:
\[\mathcal{W}_{H}^{-}(\phi):=\{f=h+\bar{g} \in \mathcal{H}:h-g=\phi\},\]
and
\[\mathcal{W}_{H}^{+}(\phi):=\{f=h+\bar{g} \in \mathcal{H}:h+g=\phi\}.\]
Dorff \cite{dorff1} proved that if $f_1,f_2 \in \mathcal{W}_{H}^{+}(z/(1-z))$ and $f_1$, $f_2$, $f_1*f_2$ are all sense-preserving in $\mathbb{D}$ then $f_1*f_2\in \mathcal{S}_{H}^{0}$ and is convex in the direction of real axis. Dorff et al.\ \cite{dorff2} established the cases under which the assumption of $f_1*f_2$ to be sense-preserving can be omitted. In this paper, we shall investigate the convolution properties of the classes $\mathcal{W}_{H}^{-}(\phi)$ and $\mathcal{W}_{H}^{+}(\phi)$, for specific choices of $\phi$.

In Section \ref{sec2}, we shall determine the conditions under which the harmonic convolution $f_1*f_2$ is univalent and convex in one direction if $f_1 \in \mathcal{W}_{H}^{-}(z)$ and $f_2 \in \mathcal{W}_{H}^{-}(\phi)$. Although the functions in these classes need not be univalent in $\mathbb{D}$, but it has been shown that their convolution is univalent in $\mathbb{D}$ and convex in the direction of real axis, under certain milder conditions. It turns out to be a simple, but an effective tool in the construction of univalent harmonic mappings convex in a given direction.

A similar analysis is carried out in Section \ref{sec3} to determine the conditions for the harmonic convolution $f_1*f_2$ to be univalent and convex in the direction of imaginary axis if $f_1 \in \mathcal{W}_{H}^{-}(z)$ and $f_2 \in \mathcal{W}_{H}^{+}(\phi)$.

\section{Univalence and Convexity in the direction of real axis}\label{sec2}

The first theorem of this section gives a sufficient condition for univalence and convexity in the direction of real axis of the harmonic convolution $f_1*f_2$ if $f_1 \in \mathcal{W}_{H}^{-}(z)$ and $f_2 \in \mathcal{W}_{H}^{-}(\phi)$.

\begin{theorem}\label{th2.1}
Let $f_1=h_1+\overline{g}_1\in \mathcal{W}_{H}^{-}(z)$ and $f_2=h_2+\overline{g}_2 \in \mathcal{W}_{H}^{-}(\phi)$. Then
\begin{enumerate}
  \item [(i)] $f_1*f_2 \in \mathcal{W}_{H}^{-}(h_1*\phi)$;
  \item [(ii)] If the analytic function $h_1*\phi$ is univalent and convex in the direction of real axis and
  \[\RE \frac{(h_1*h_2)'(z)}{(h_1*\phi)'(z)}>\frac{1}{2},\quad z\in \mathbb{D}\]
  then $f_1*f_2 \in S_{H}^{0}$ and is convex in the direction of real axis.
\end{enumerate}
\end{theorem}

\begin{proof}
Recall that $\psi *z=z=z*\psi$ for any analytic function $\psi$ with $\psi'(0)=1$, so we can write
\begin{align*}
z=z*(h_2+g_2)&=(h_1-g_1)*(h_2+g_2)\\
             &=(h_1*h_2)+(h_1*g_2)-(h_2*g_1)-(g_1*g_2)
\end{align*}
and
\begin{align*}
(h_1+g_1)*\phi&=(h_1+g_1)*(h_2-g_2)\\
             &=(h_1*h_2)-(h_1*g_2)+(h_2*g_1)-(g_1*g_2).
\end{align*}
Thus, we obtain
\begin{align*}
(h_1*h_2)-(g_1*g_2)&=\frac{1}{2}[z+(h_1+g_1)*\phi]\\
                   &=\frac{1}{2}[z+2(h_1*\phi)-z*\phi]=h_1*\phi,
\end{align*}
since $h_1-g_1=z$. This proves (i).

For the proof of (ii), note that $(h_1*h_2)-(g_1*g_2)$ is univalent and convex in the direction of real axis. In order to apply Lemma \ref{lem1} to the function $f_1*f_2$, we need to show that the dilatation $w_{f_1*f_2}=(g_1*g_2)'/(h_1*h_2)'$ of $f_1*f_2$ satisfies $|w_{f_1*f_2}(z)|<1$ for all $z \in \mathbb{D}$. Equivalently, it suffices to show that $\RE ((1+w_{f_1*f_2})/(1-w_{f_1*f_2}))>0$ in $\mathbb{D}$. Using the identity $h_1*\phi=(h_1*h_2)-(g_1*g_2)$, we see that
\begin{align*}
\RE \left(\frac{1+w_{f_1*f_2}}{1-w_{f_1*f_2}}\right)&=\RE \frac{(h_1*h_2)'+(g_1*g_2)'}{(h_1*h_2)'-(g_1*g_2)'}\\
                                                    &=\RE \frac{(h_1*h_2)'+(g_1*g_2)'}{(h_1*\phi)'}\\
                                                    &=2\RE \frac{(h_1*h_2)'}{(h_1*\phi)'}-1
\end{align*}
which is clearly positive, under the hypothesis of the theorem.
\end{proof}

Making use of Theorem \ref{th2.1}, we will investigate the convolution properties of functions in the class $\mathcal{W}_{H}^{-}(z)$. Taking $\phi(z)\equiv z$, we obtain
\begin{corollary}\label{cor2.2}
Let both $f_1=h_1+\overline{g}_1, f_2=h_2+\overline{g}_2 \in \mathcal{W}_{H}^{-}(z)$. Then
\begin{itemize}
  \item [(i)] $f_1*f_2 \in \mathcal{W}_{H}^{-}(z)$. In particular, the class $\mathcal{W}_{H}^{-}(z)$ is closed under harmonic convolution;
  \item [(ii)] If $\RE (h_1*h_2)'>1/2$ in $\mathbb{D}$, then $f_1*f_2 \in S_{H}^{0}$ and is convex in the direction of real axis.
\end{itemize}
\end{corollary}

It is easy to see that the range of $f \in \mathcal{W}_{H}^{-}(z)$ is contained in the horizontal strip $|\IM w|<1$. Corollary \ref{cor2.2} shows that the same is true for the harmonic convolution $f_1*f_2$ if $f_1,f_2\in \mathcal{W}_{H}^{-}(z)$. Now, we provide some examples to illustrate Corollary \ref{cor2.2}.

\begin{example}\label{ex2.3}
For $n=2,3,\ldots$, let $p_n=u_n+\overline{v}_n$ be the harmonic mappings of $\mathbb{D}$ with $u_n=z+z^n/n$ and $v_n=z^n/n$. Then $p_n \in \mathcal{W}_{H}^{-}(z)$ and
\[p_n*p_n=u_n*u_n+\overline{v_n*v_n}=z+\frac{z^n}{n^2}+\frac{\overline{z}^n}{n^2}\]
for $n=2,3,\ldots$. For $z \in \mathbb{D}$, $\RE (u_n*u_n)'(z)=1+\RE z^{n-1}/n>1-1/n\geq 1/2$ which imply that the convolution maps $p_n*p_n$ $(n=2,3,\ldots)$ are univalent and convex in the direction of real axis, by Corollary \ref{cor2.2}. The images of the unit disk under $p_n$ and $p_n*p_n$ for $2\leq n\leq 4$ are shown in Figure \ref{fig1} as plots of the images of equally spaced radial segments and concentric circles.

\begin{figure}[htb!]
\begin{center}
  \subfigure[$p_2$]{\includegraphics[width=2.5in,height=1.5in]{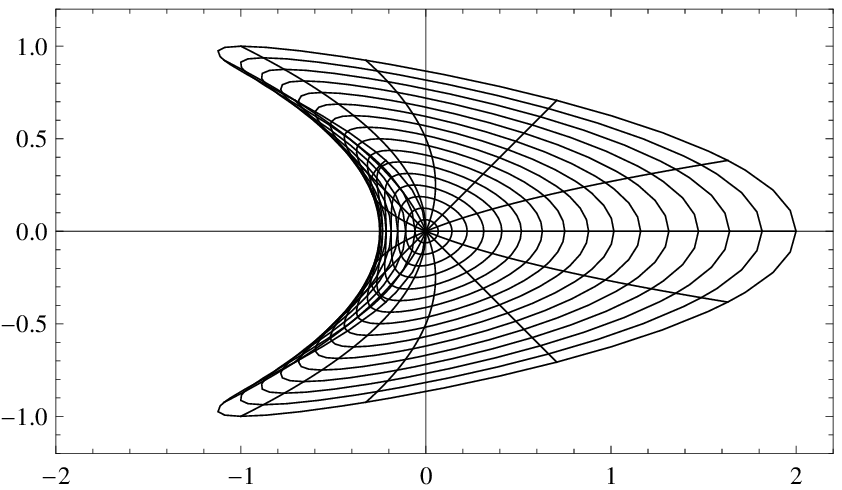}}\hspace{5pt}
  \subfigure[$p_2*p_2$]{\includegraphics[width=2.5in,height=1.5in]{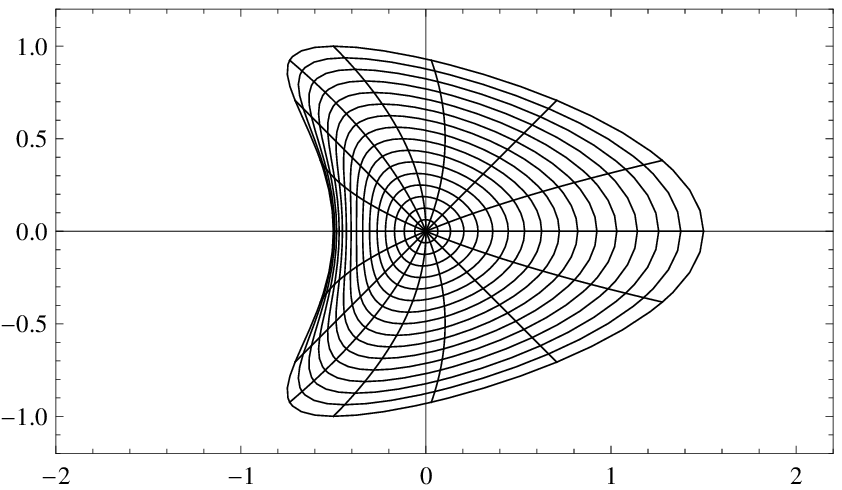}}\hspace{5pt}
  \subfigure[$p_3$]{\includegraphics[width=2.5in,height=1.5in]{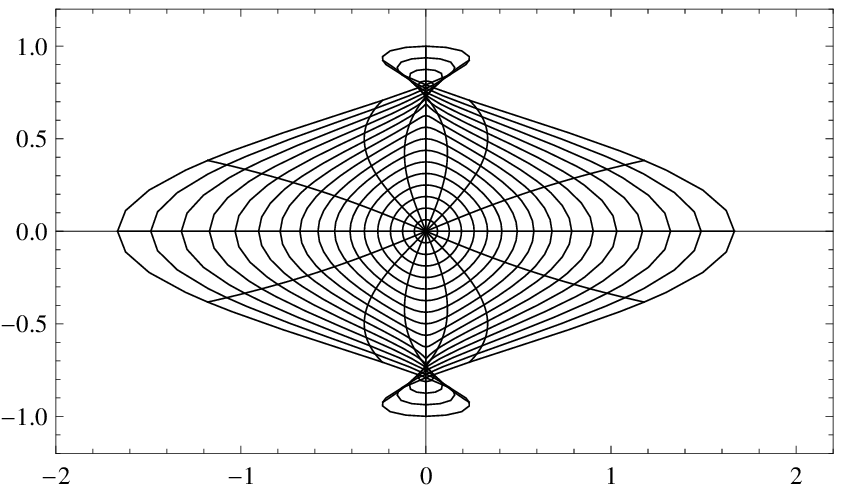}}\hspace{5pt}
  \subfigure[$p_3*p_3$]{\includegraphics[width=2.5in,height=1.5in]{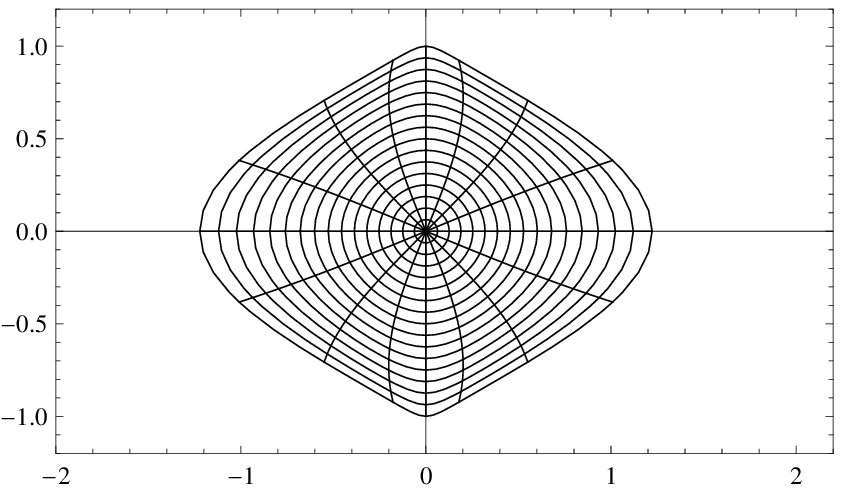}}\hspace{5pt}
  \subfigure[$p_4$]{\includegraphics[width=2.5in,height=1.5in]{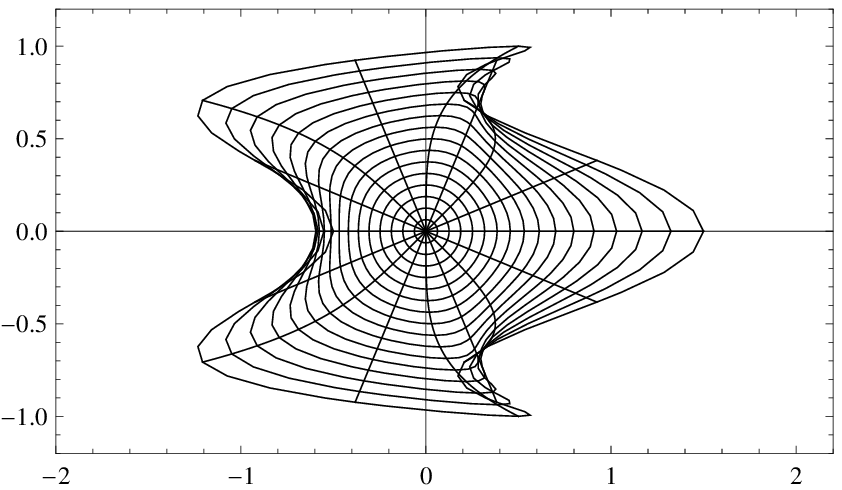}}\hspace{5pt}
  \subfigure[$p_4*p_4$]{\includegraphics[width=2.5in,height=1.5in]{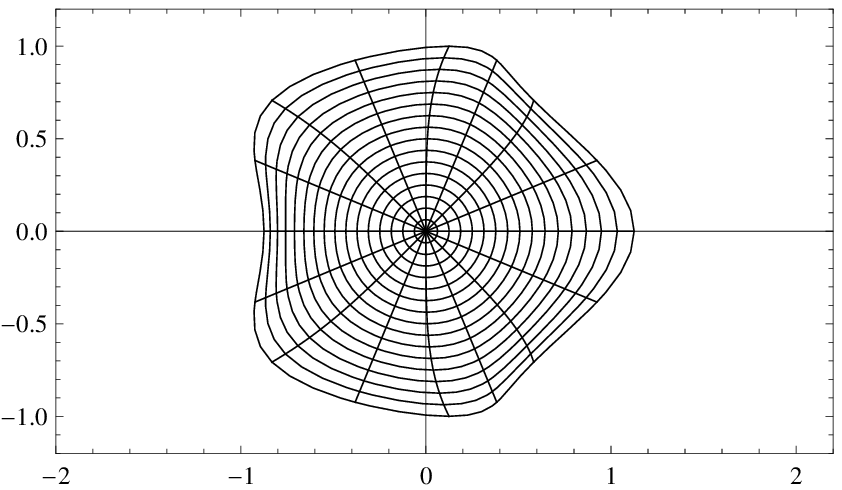}}
  \caption{Images of the functions $p_n(z)=z+z^n/n+\bar{z}^n/n$ and their convolutions $p_n*p_n$ for $n=2,3,4$.}\label{fig1}
\end{center}
\end{figure}

Note that the mappings $p_n\in \mathcal{W}_{H}^{-}(z)$ $(n=2,3,\ldots)$ are not univalent in $\mathbb{D}$ since their Jacobian $J_{p_n}(z)=1+2\RE z^{n-1}$ vanishes inside $\mathbb{D}$.
\end{example}

\begin{remark}\label{rem2.4}
The images of the convolution maps $p_n*p_n$ (see Figure \ref{fig1}) in Example \ref{ex2.3} suggest that $p_n*p_n \not\in \mathcal{K}_{H}^{0}$. Thus, the conclusion of the Corollary \ref{cor2.2} can't to strengthened to $f_1*f_2 \in \mathcal{K}_{H}^{0}$.
\end{remark}

\begin{example}\label{ex2.5}
For $k=1,2,\ldots$, let $\Gamma_k=\mu_k+\overline{\nu}_k$ be the shears of the identity map in the direction of real axis with dilatation $w_{\Gamma_k}(z)=z^k$. Then $\Gamma_k \in \mathcal{W}_H^-(z)$ and
\[\mu_k(z)=z+\sum_{n=1}^{\infty}\frac{z^{nk+1}}{nk+1},\quad \nu_k(z)=\sum_{n=1}^{\infty}\frac{z^{nk+1}}{nk+1} \quad (z \in \mathbb{D};k=1,2,\ldots).\]
In particular, we have
\begin{equation}\label{eq2.1}
\Gamma_1=\mu_1+\overline{\nu}_1,\quad  \mu_1(z)=-\log(1-z), \quad \nu_1(z)=-z-\log(1-z);
\end{equation}
and
\[\Gamma_2=\mu_2+\overline{\nu}_2,\quad  \mu_2(z)=\frac{1}{2}\log \left(\frac{1+z}{1-z}\right), \quad \nu_2(z)=-z+\frac{1}{2}\log \left(\frac{1+z}{1-z}\right).\]

\begin{figure}[here]
\begin{center}
  \subfigure[$\Gamma_1$]{\includegraphics[width=2.5in,height=1.5in]{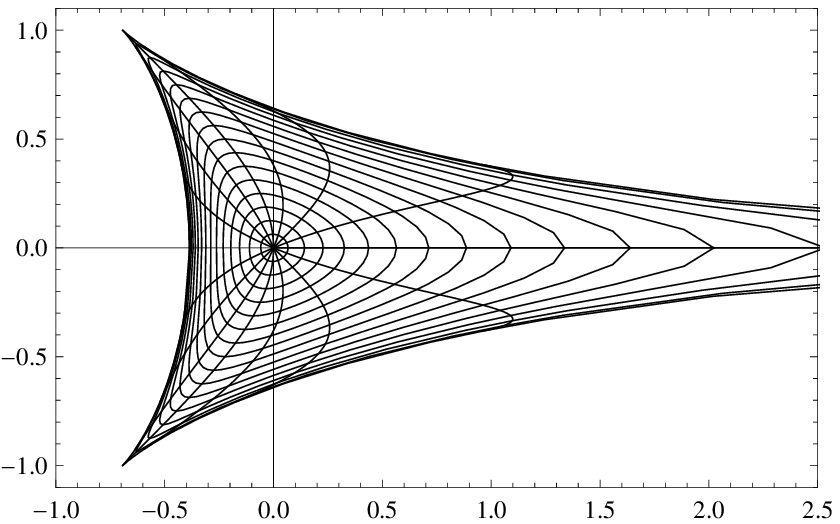}}\hspace{5pt}
  \subfigure[$\Gamma_1*\Gamma_1$]{\includegraphics[width=2.5in,height=1.5in]{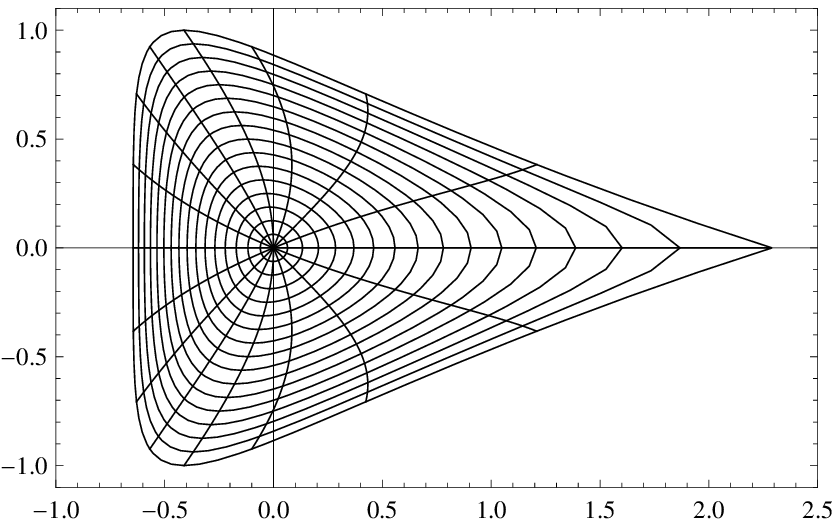}}\hspace{5pt}
  \subfigure[$\Gamma_2$]{\includegraphics[width=2.5in,height=1.5in]{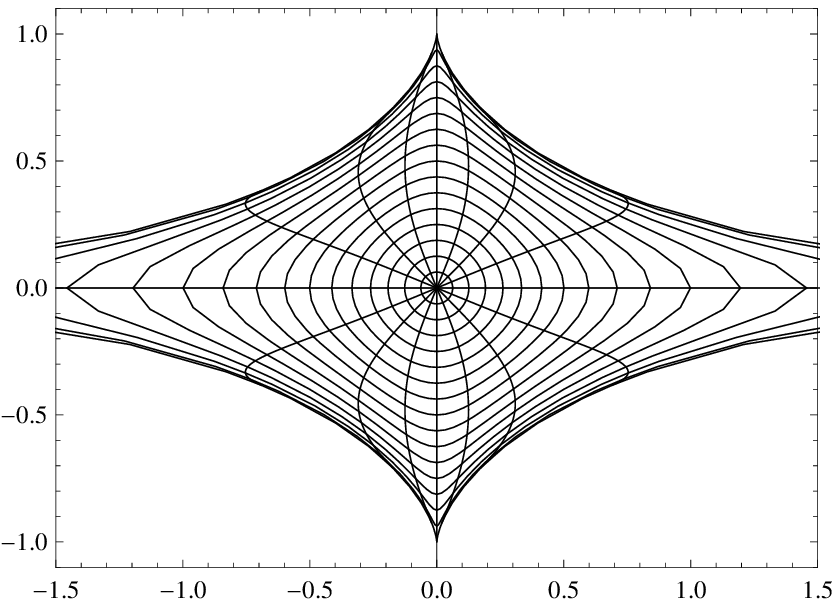}}\hspace{5pt}
  \subfigure[$\Gamma_2*\Gamma_2$]{\includegraphics[width=2.5in,height=1.5in]{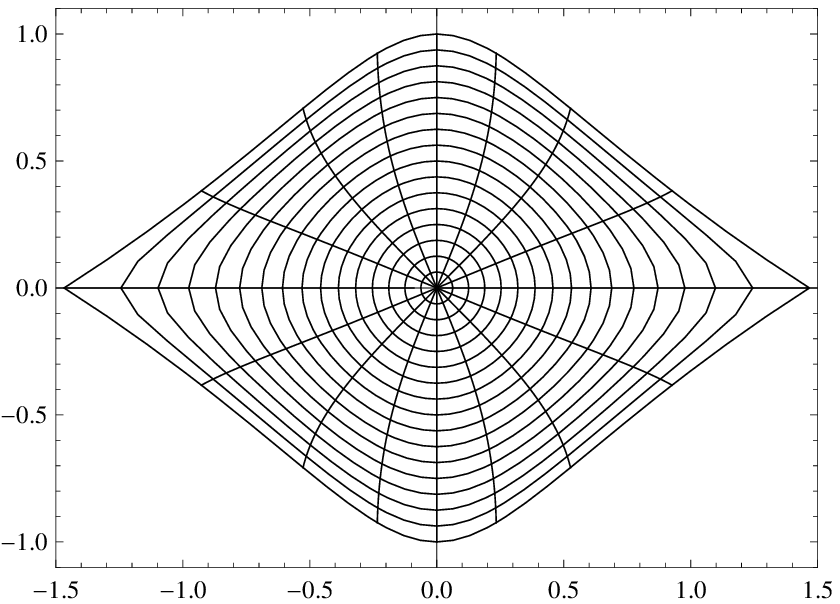}}\hspace{5pt}
  \subfigure[$\Gamma_3$]{\includegraphics[width=2.5in,height=1.5in]{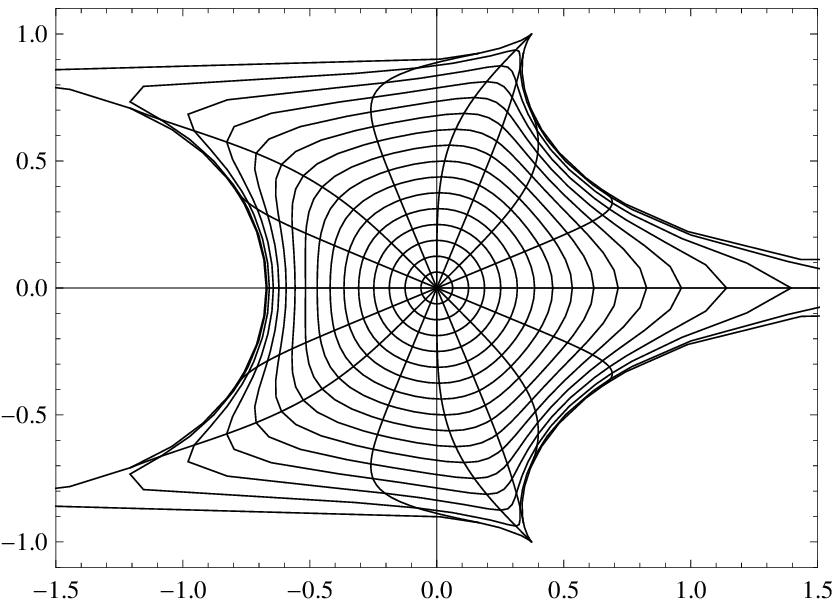}}\hspace{5pt}
  \subfigure[$\Gamma_3*\Gamma_3$]{\includegraphics[width=2.5in,height=1.5in]{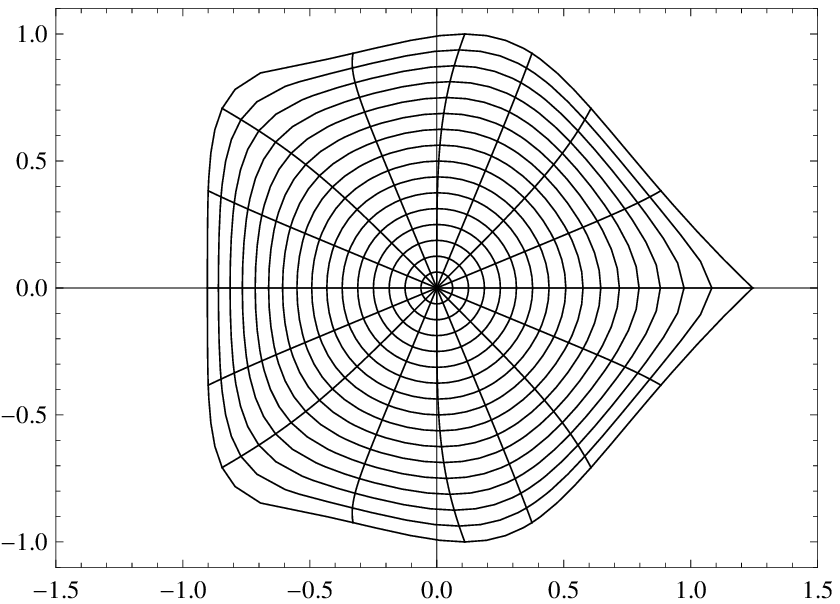}}
  \caption{Images of the functions $\Gamma_k$ and their convolutions $\Gamma_k*\Gamma_k$ for $k=1,2,3$.}\label{fig2}
\end{center}
\end{figure}

To apply Corollary \ref{cor2.2} for the functions $\Gamma_k*\Gamma_k$, we need to show that $\RE (\mu_k*\mu_k)'>1/2$ ($k=1,2,\ldots$) in $\mathbb{D}$. Note that
\[(\mu_k*\mu_k)(z)=z+\sum_{n=1}^{\infty}\frac{z^{nk+1}}{(nk+1)^2},\quad z \in \mathbb{D}.\]
Differentiation with respect to $z$ leads to an expression
\[z(\mu_k*\mu_k)''(z)+(\mu_k*\mu_k)'(z)=1+\sum_{n=1}^{\infty}z^{nk}=\frac{1}{1-z^k},\quad z \in \mathbb{D}.\]
This shows that
\[\RE (z(\mu_k*\mu_k)''(z)+(\mu_k*\mu_k)'(z))>1/2,\quad z\in \mathbb{D}.\]
By \cite[Theorem 2]{silverman}, it follows that $\RE (\mu_k*\mu_k)'>\log 2\simeq 0.69314$ in $\mathbb{D}$. Now, Corollary \ref{cor2.2} shows that the convolution maps
\[(\Gamma_k*\Gamma_k)(z)=z+\sum_{n=1}^{\infty}\frac{z^{nk+1}}{(nk+1)^2}+\overline{\sum_{n=1}^{\infty}\frac{z^{nk+1}}{(nk+1)^2}} \quad (z \in \mathbb{D},k=1,2,\ldots)\]
are univalent and convex in the direction of real axis. Figure \ref{fig2} depicts the image domains $\Gamma_k(\mathbb{D})$ and $(\Gamma_k*\Gamma_k)(\mathbb{D})$ for $k=1,2,3$.
\end{example}

The next theorem deals with the convolution of $\Gamma_1 \in \mathcal{W}_{H}^{-}(z)$ given by \eqref{eq2.1} with harmonic mappings $f \in \mathcal{W}_{H}^{-}(\phi)$.

\begin{theorem}\label{th2.6}
Let $f=h+\overline{g} \in \mathcal{W}_{H}^{-}(\phi)$ with $\RE h(z)/\phi(z)>1/2$ for all $z \in \mathbb{D}$. If the analytic function $\mu_1*\phi$ is univalent and convex in the direction of real axis, then $\Gamma_1*f \in \mathcal{S}_{H}^{0}$ and is convex in the direction of real axis, $\Gamma_1=\mu_1+\overline{\nu}_1$ is given by \eqref{eq2.1}.
\end{theorem}

\begin{proof}
Note that $z(\mu_1*h)'=z\mu'_1*h=z/(1-z)*h=h$. Similarly, $z(\mu_1*\phi)'=\phi$. Therefore $\RE (\mu_1*h)'/(\phi*h)'=\RE h/\phi>1/2$ in $\mathbb{D}$. The result now follows by applying Theorem \ref{th2.1}.
\end{proof}

\begin{example}\label{ex2.7}
Consider the non-univalent harmonic function $f=h+\overline{g}$ where $h(z)=z(1+z)/(1-z)^2$ and $g(z)=z^2(1+z)/(1-z)^2$. Then $f \in \mathcal{W}_{H}^{-}(\phi)$, $\phi(z)=z(1+z)/(1-z)$. Note that $\RE h(z)/\phi(z)=\RE 1/(1-z)>1/2$ and $(\mu_1*\phi)(z)=-z-2\log (1-z)$ is univalent and convex in the direction of real axis. By Theorem \ref{th2.6}, the convolution
\[(\Gamma_1*f)(z)=\frac{2z}{1-z}+\log (1-z)+\overline{\frac{3z-z^2}{1-z}+3\log (1-z)},\quad z\in \mathbb{D}.\]
is univalent and convex in the direction of real axis (see Figure \ref{fig3}).

\begin{figure}[here]
\begin{center}
  \subfigure{\includegraphics[width=2.8in,height=1.8in]{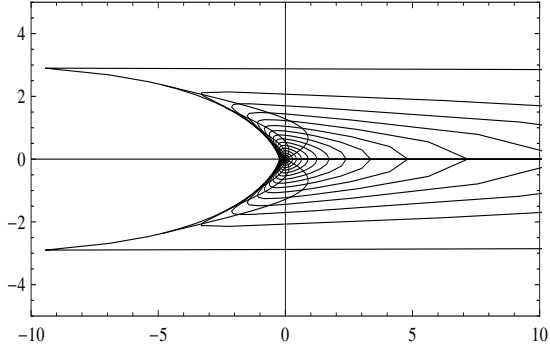}}
  \caption{Image of the convolution of $\Gamma_1$ with $f(z)=z(1+z)/(1-z)^2+\overline{z}^2(1+\overline{z})/(1-\overline{z})^2$.}\label{fig3}
\end{center}
\end{figure}
\end{example}

In \cite{sumit3}, the authors introduced the notion of positive harmonic Alexander operator $\Lambda_{H}^{+}:\mathcal{H}\rightarrow\mathcal{H}$ defined by
\begin{equation}\label{eq2.2}
\Lambda_{H}^{+}[f](z)=\int_0^z \frac{h(t)}{t}\,dt+\overline{\int_0^z \frac{g(t)}{t}\,dt},\quad (f=h+\bar{g}\in \mathcal{H};z\in \mathbb{D}).
\end{equation}
In general, $\Lambda_{H}^{+}[\mathcal{S}_{H}^{0}]\not\subset \mathcal{S}_{H}^{0}$. Since $\Lambda_{H}^{+}[f]=f*\Gamma_1$ where $\Gamma_1$ is given by \eqref{eq2.1}, Theorem \ref{th2.6} determines a class of harmonic mappings that is mapped into $\mathcal{S}_{H}^{0}$ by the positive harmonic Alexander operator $\Lambda_{H}^{+}$. For specific choices of $\phi$ in Theorem \ref{th2.6}, we obtain the following corollary.

\begin{corollary}\label{cor2.8}
Suppose that $f=h+\overline{g} \in \mathcal{H}$ satisfies one of the following three conditions:
\begin{enumerate}
  \item [(i)] $f \in \mathcal{W}_{H}^{-}(z)$ and $\RE h(z)/z>1/2$ $(z\in \mathbb{D})$;
  \item [(ii)] $f \in \mathcal{W}_{H}^{-}(z/(1-z))$ and $\RE (1-z)h(z)/z>1/2$ $(z\in \mathbb{D})$;
  \item [(iii)] $f \in \mathcal{W}_{H}^{-}(z/(1-z)^2)$ and $\RE (1-z)^2h(z)/z>1/2$ $(z\in \mathbb{D})$.
\end{enumerate}
Then $\Lambda_{H}^{+}[f] \in \mathcal{S}_{H}^{0}$ and is convex in the direction of real axis.
\end{corollary}

Suppose that $f=h+\overline{g} \in \mathcal{K}_{H}^{0}$. Then $|h(z)|>|g(z)|$ for all $z \in \mathbb{D}\backslash \{0\}$ by \cite[Corollary 5.8]{cluniesheilsmall}. If, in addition, $h(z)-g(z)=z$ then $|h(z)|>|h(z)-z|$ which imply that $|h(z)/z|>|h(z)/z-1|$ for all $z \in \mathbb{D}\backslash \{0\}$. This shows that $\RE h(z)/z>1/2$ for $z \in \mathbb{D}\backslash \{0\}$. The inequality $\RE h(z)/z>1/2$ is obviously true if $z=0$. Therefore Corollary \ref{cor2.8}(i) gives

\begin{corollary}\label{cor2.9}
If $f \in \mathcal{W}_{H}^{-}(z) \cap \mathcal{K}_{H}^{0}$ then $\Lambda_{H}^{+}[f] \in S_{H}^{0}$ and is convex in the direction of real axis.
\end{corollary}

\begin{example}\label{ex2.10}
The function $f(z)=z+z^2/8+\overline{z}^2/8 \in \mathcal{W}_{H}^{-}(z)$ is sense-preserving and satisfies
\[\frac{\partial}{\partial \theta}\left(\arg \left\{\frac{\partial}{\partial \theta}f(r e^{i \theta})\right\}\right)>0,\quad 0\leq\theta<2\pi,\quad 0<r<1.\]
This shows that $f$ is fully convex in $\mathbb{D}$ (see \cite{mocanu,sumit2}) and hence $f \in \mathcal{K}_{H}^{0}$. By Corollary \ref{cor2.9}, the convolution $(\Gamma_1*f)(z)=z+z^2/16+\overline{z}^2/16$ is univalent in $\mathbb{D}$ and convex in the direction of real axis (see Figure \ref{fig4}).
\begin{figure}[here]
\begin{center}
  \subfigure[$f$]{\includegraphics[width=2.5in,height=1.8in]{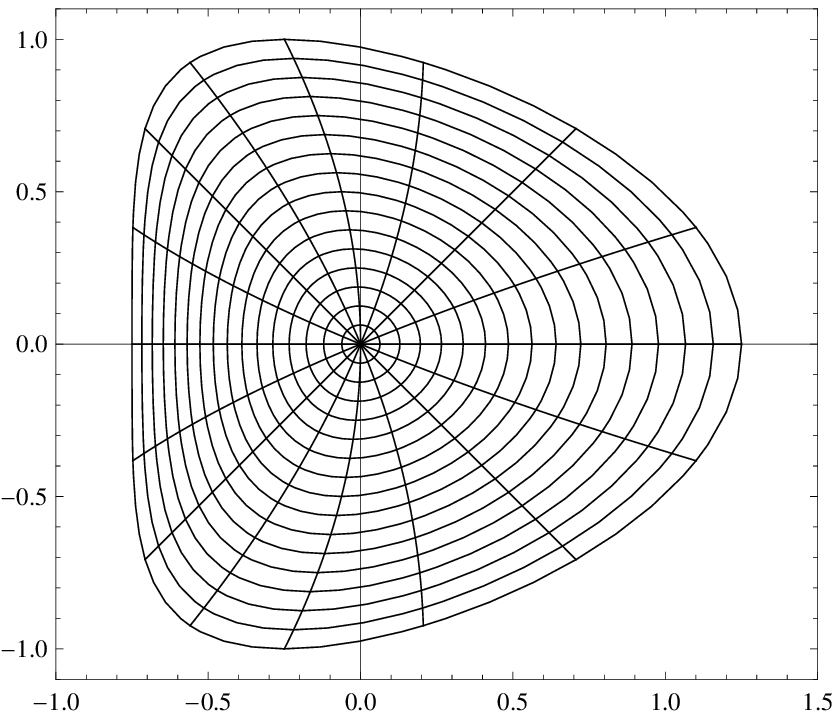}}\hspace{5pt}
  \subfigure[$\Gamma_1*f$]{\includegraphics[width=2.5in,height=1.8in]{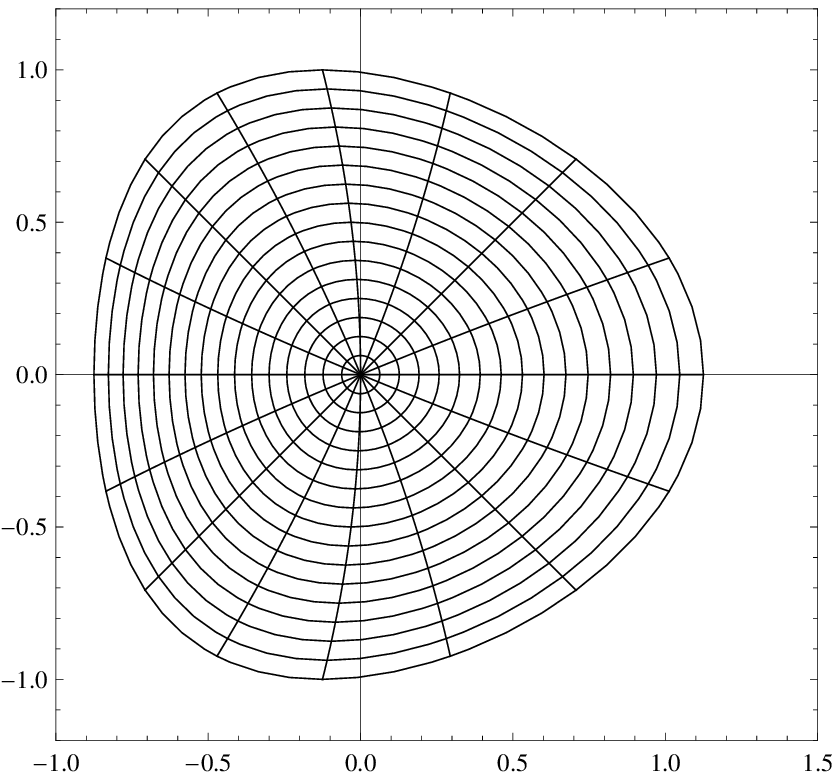}}
  \caption{Images of the function $f(z)=z+z^2/8+\overline{z}^2/8$ and its convolution with $\Gamma_1$.}\label{fig4}
\end{center}
\end{figure}
\end{example}

\begin{corollary}\label{cor2.11}
If $f=h+\overline{g} \in \mathcal{W}_{H}^{-}(z)$ with $h \in \mathcal{K}$, then $\Lambda_{H}^{+}[f] \in S_{H}^{0}$ and is convex in the direction of real axis.
\end{corollary}

\begin{proof}
Since $h\in \mathcal{K}$, $\RE h(z)/z>1/2$ for $z \in \mathbb{D}$ by the well-known Marx Strohh\"{a}cker theorem \cite[Theorem 2.6(a), p. 57]{monograph}. By Corollary \ref{cor2.8}(i), we obtain the desired result.
\end{proof}

Since $\mu_1 \in \mathcal{K}$, Corollary \ref{cor2.11} directly shows that the convolution $\Gamma_1*\Gamma_1 \in \mathcal{S}_{H}^{0}$ and is convex in the direction of real axis, where $\Gamma_1=\mu_1+\overline{\nu}_1\in \mathcal{W}_{H}^{-}(z)$ is given by \eqref{eq2.1}. Parts $(ii)$ and $(iii)$ of Corollary \ref{cor2.8} are illustrated by the following examples.

\begin{example}\label{ex2.12}
The harmonic mapping
\[F(z)=U(z)+\overline{V(z)},\quad  U(z):=\frac{z-\frac{1}{2}z^{2}}{(1-z)^2},\quad  V(z):=\frac{\frac{1}{2}z^2}{(1-z)^2}\]
constructed by shearing the conformal mapping $l(z)=z/(1-z)$ in the direction of real axis with dilatation $w_F(z)=z$, belongs to $\mathcal{W}_{H}^{-}(z/(1-z))$ and $F(\mathbb{D})=\{u+iv:v^2>-(u+1/4)\}$ (see Figure \ref{fig5}(a)). Recently, the authors \cite{sumit4} calculated the radius of starlikeness and convexity of the function $F$. Note that $\RE (1-z)U(z)/z=\RE (2-z)/(2(1-z))>3/4$ for all $z \in \mathbb{D}$. By Corollary \ref{cor2.8}(ii), the convolution
\[(\Gamma_1*F)(z)=\RE \frac{z}{1-z}-i \arg\{1-z\},\quad z\in \mathbb{D}\]
is univalent and convex in the direction of real axis (see Figure \ref{fig5}(b)).

\begin{figure}[here]
\begin{center}
  \subfigure[$F$]{\includegraphics[width=2.5in,height=2.5in]{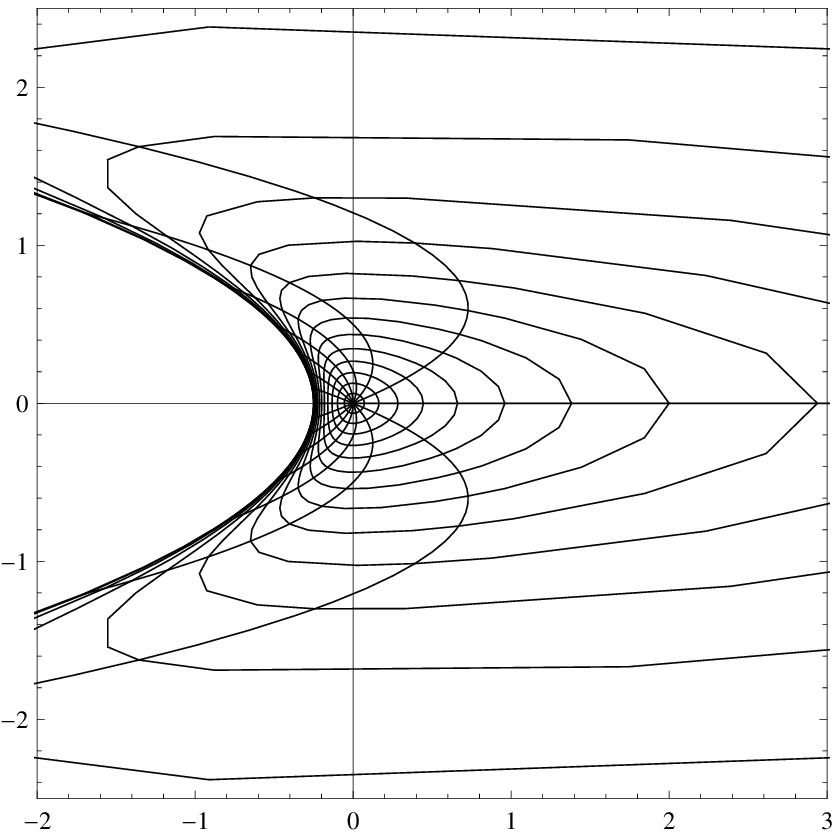}}\hspace{5pt}
  \subfigure[$\Gamma_1*F$]{\includegraphics[width=2.5in,height=2.5in]{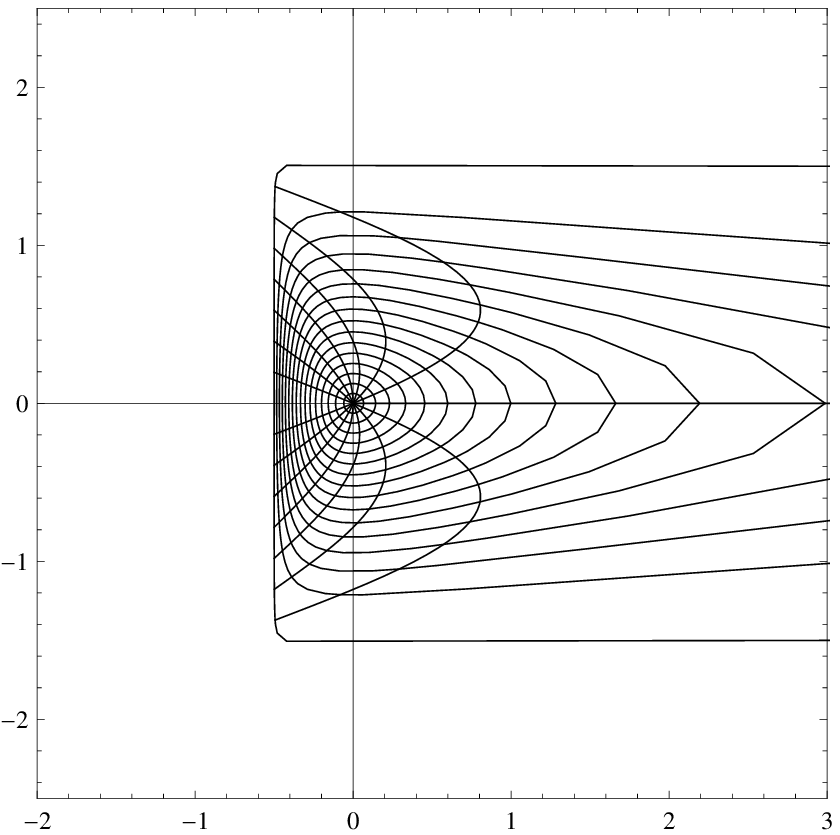}}
\caption{Images of the function $F$ and its convolution with $\Gamma_1$}\label{fig5}
\end{center}
\end{figure}
\end{example}

\begin{example}\label{ex2.13}
The harmonic Koebe function
\[K(z)=H(z)+\overline{G(z)}, \quad H(z):=\frac{z-\frac{1}{2}z^2+\frac{1}{6}z^3}{(1-z)^3},\quad G(z):=\frac{\frac{1}{2}z^2+\frac{1}{6}z^3}{(1-z)^3},\]
constructed by shearing the Koebe function $k(z)=z/(1-z)^2$ in the direction of real axis with dilatation $w_K(z)=z$, maps the unit disk $\mathbb{D}$ onto the slit-plane $\mathbb{C}\backslash(-\infty,-1/6]$. Note that $K \in \mathcal{W}_{H}^{-}(z/(1-z)^2)$ and $\RE (1-z)^2 H(z)/z>1/2$ so the convolution
\[\Lambda_{H}^{+}[K]=(\Gamma_1*K)(z)=\frac{2}{3}\frac{z}{(1-z)^2}+\frac{1}{3}i \IM \frac{z-3z^2}{(1-z)^2}-\frac{1}{3}\log |1-z|,\quad z\in \mathbb{D}\]
is univalent and convex in the direction of real axis (see Figure \ref{fig6}) by Corollary \ref{cor2.8}(iii).
\begin{figure}
\begin{center}
  \subfigure{\includegraphics[width=2.8in,height=3 in]{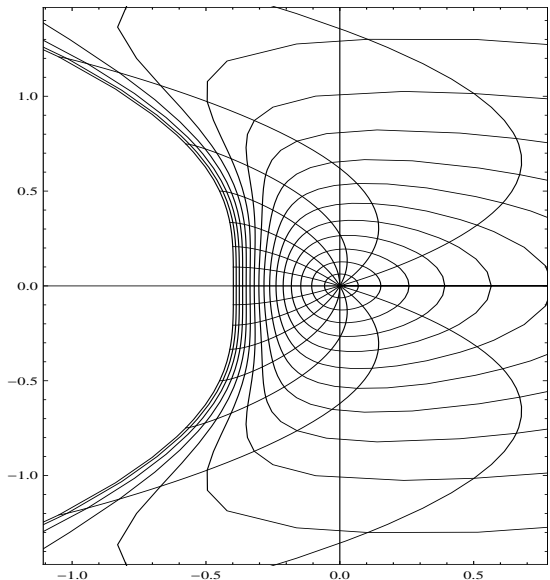}}
  \caption{Image of the convolution of $\Gamma_1$ with the harmonic Koebe function $K$.}\label{fig6}
\end{center}
  \end{figure}
\end{example}

Analogous to Theorem \ref{th2.1}, the last theorem of this section determines the conditions under which the harmonic convolution $f_1*f_2$ is univalent and convex in one direction if $f_1 \in \mathcal{W}_{H}^{+}(z)$ and $f_2 \in \mathcal{W}_{H}^{+}(\phi)$. Its proof follows by an easy modification of the proof of Theorem \ref{th2.1}.

\begin{theorem}\label{th2.14}
Let $f_1=h_1+\overline{g}_1\in \mathcal{W}_{H}^{+}(z)$ and $f_2=h_2+\overline{g}_2 \in \mathcal{W}_{H}^{+}(\phi)$. Then $f_1*f_2 \in \mathcal{W}_{H}^{-}(h_1*\phi)$. Moreover, if $h_1*\phi \in \mathcal{S}$ and is convex in the direction of real axis with $\RE(h_1*h_2)'/(h_1*\phi)'>1/2$ in $\mathbb{D}$ then $f_1*f_2 \in S_{H}^{0}$ and is convex in the direction of real axis.
\end{theorem}

\section{Univalence and Convexity in the direction of imaginary axis}\label{sec3}
In this section, we shall investigate the properties of the convolution $f_1*f_2$ if $f_1 \in \mathcal{W}_{H}^{-}(z)$ and $f_2 \in \mathcal{W}_{H}^{+}(\phi)$.

\begin{theorem}\label{th3.1}
Let $f_1\in \mathcal{W}_{H}^{-}(z)$ and $f_2 \in \mathcal{W}_{H}^{+}(\phi)$. Then
\begin{itemize}
  \item [(i)] $f_1*f_2 \in \mathcal{W}_{H}^{+}(h_1*\phi)$;
  \item [(ii)] Further, if the analytic function $h_1*\phi$ is univalent and convex in the direction of imaginary axis and $\RE(h_1*h_2)'/(h_1*\phi)'>1/2$ in $\mathbb{D}$, then $f_1*f_2 \in S_{H}^{0}$ and is convex in the direction of imaginary axis.
\end{itemize}
\end{theorem}

\begin{proof}
Adding the two identities
\[z=(h_1-g_1)*(h_2-g_2)\quad \mbox{and} \quad (h_1+g_1)*\phi=(h_1+g_1)*(h_2+g_2),\]
we obtain $h_1*h_2+g_1*g_2=h_1*\phi$. This shows that $f_1*f_2 \in \mathcal{W}_{H}^{+}(h_1*\phi)$. Similar to the proof of Theorem \ref{th2.1}, it is easy to see that the condition $\RE(h_1*h_2)'/(h_1*\phi)'>1/2$ is equivalent to $|w_{f_1*f_2}|<1$ in $\mathbb{D}$. By applying Lemma \ref{lem1}, it follows that $f_1*f_2$ is univalent and convex in the direction of imaginary axis.
\end{proof}

Taking $\phi(z)\equiv z$ in Theorem \ref{th3.1}, we have
\begin{corollary}\label{cor3.2}
Let $f_1\in \mathcal{W}_{H}^{-}(z)$ and $f_2 \in \mathcal{W}_{H}^{+}(z)$. Then $f_1*f_2 \in \mathcal{W}_{H}^{+}(z)$ and if $\RE (h_1*h_2)'>1/2$ in $\mathbb{D}$, then $f_1*f_2 \in S_{H}^{0}$ and is convex in the direction of imaginary axis.
\end{corollary}

\begin{example}\label{ex3.3}
For $n=2,3,\ldots$, let $q_n=r_n+\overline{s}_n$ be the harmonic mappings of $\mathbb{D}$ with $r_n=z-z^n/n$ and $s_n=z^n/n$. Then $q_n \in \mathcal{W}_{H}^{+}(z)$ are not univalent in $\mathbb{D}$ and
\[p_n*q_n=u_n*r_n+\overline{v_n*s_n}=z-\frac{z^n}{n^2}+\frac{z^n}{n^2}\quad (n=2,3,\ldots)\]
where $p_n=u_n+\overline{v}_n\in \mathcal{W}_{H}^{-}(z)$ are defined in Example \ref{ex2.3}.

\begin{figure}[here]
\begin{center}
  \subfigure[$q_2$]{\includegraphics[width=1.5in,height=2in]{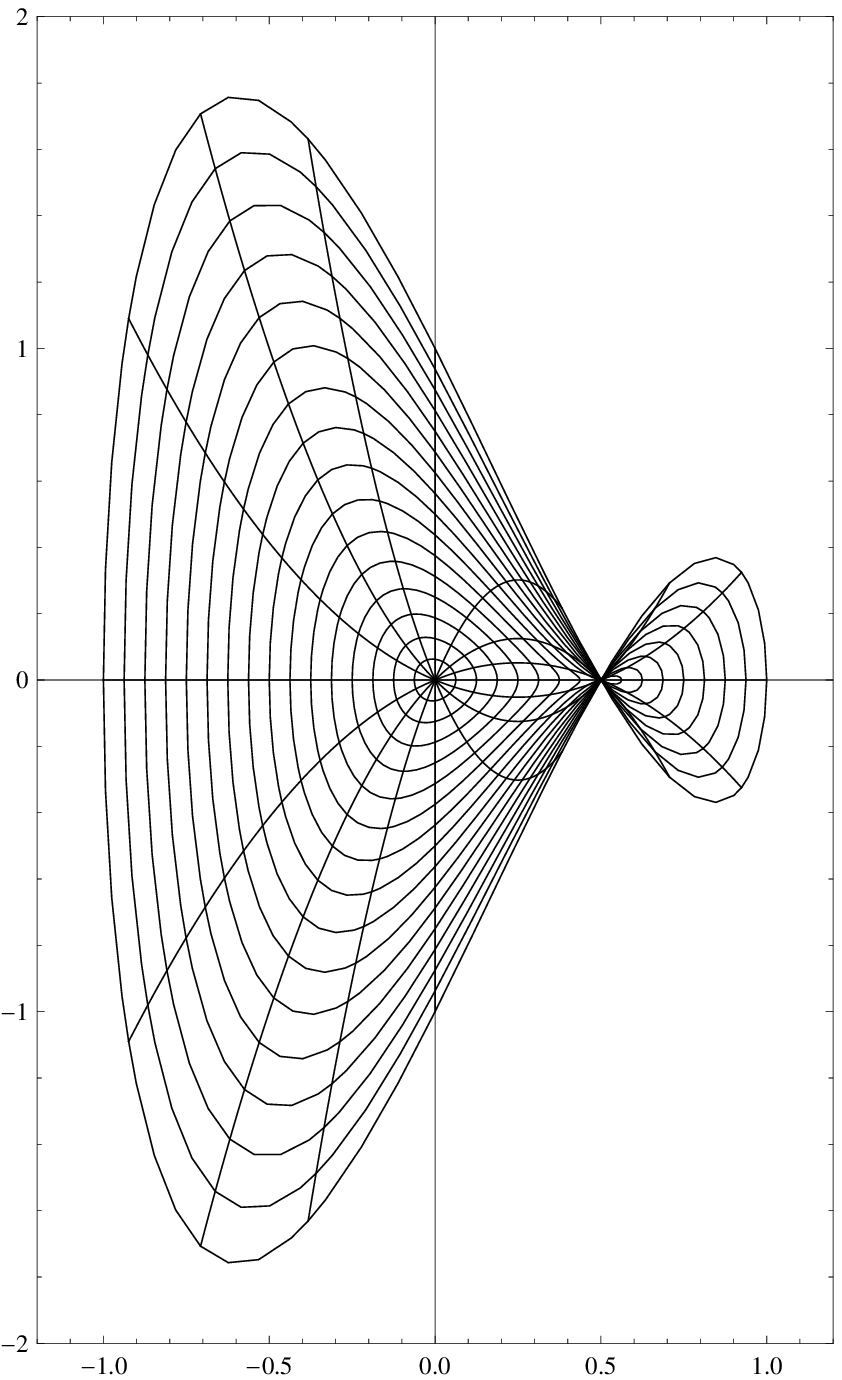}}\hspace{5pt}
  \subfigure[$q_3$]{\includegraphics[width=1.5in,height=2in]{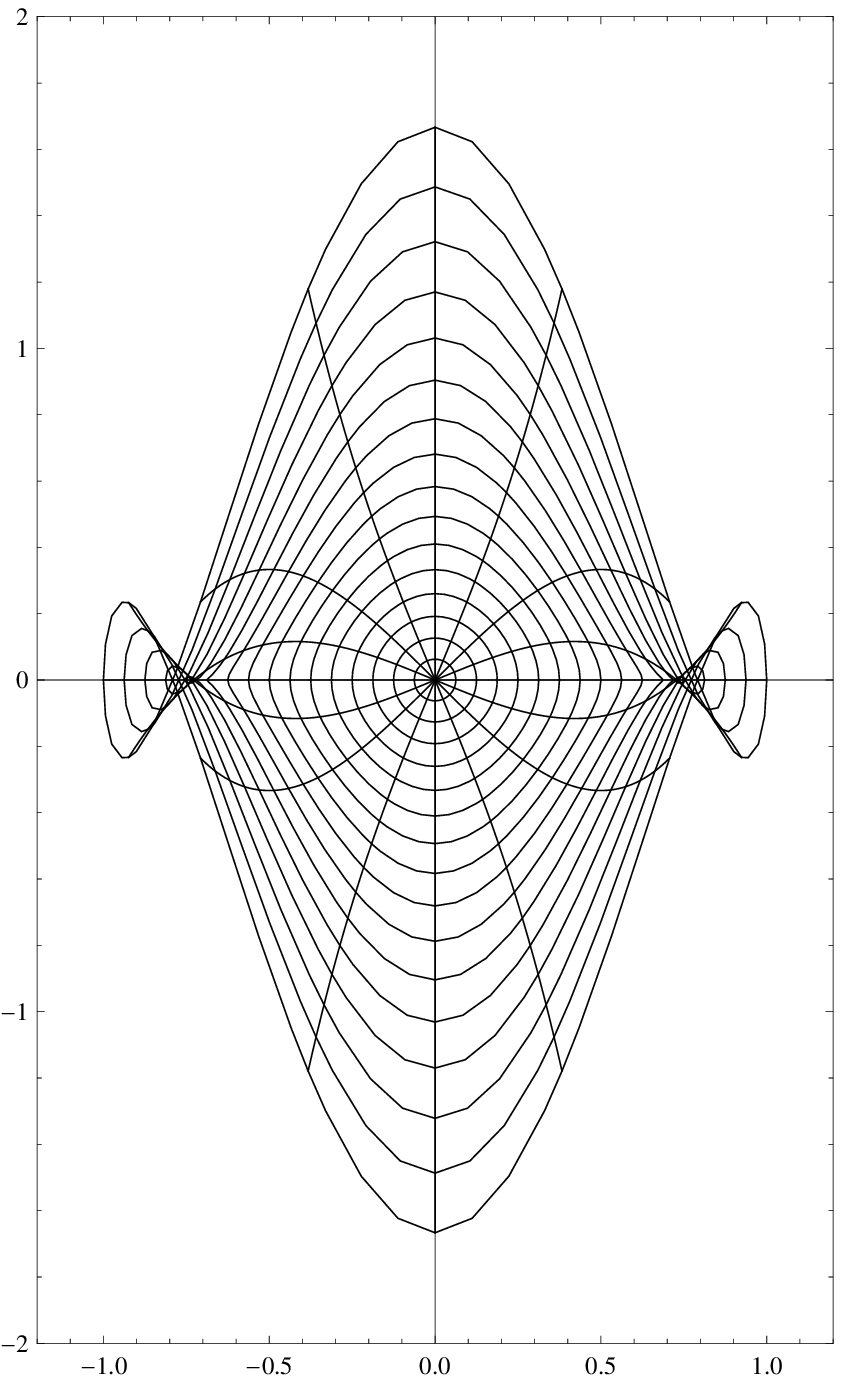}}\hspace{5pt}
  \subfigure[$q_4$]{\includegraphics[width=1.5in,height=2in]{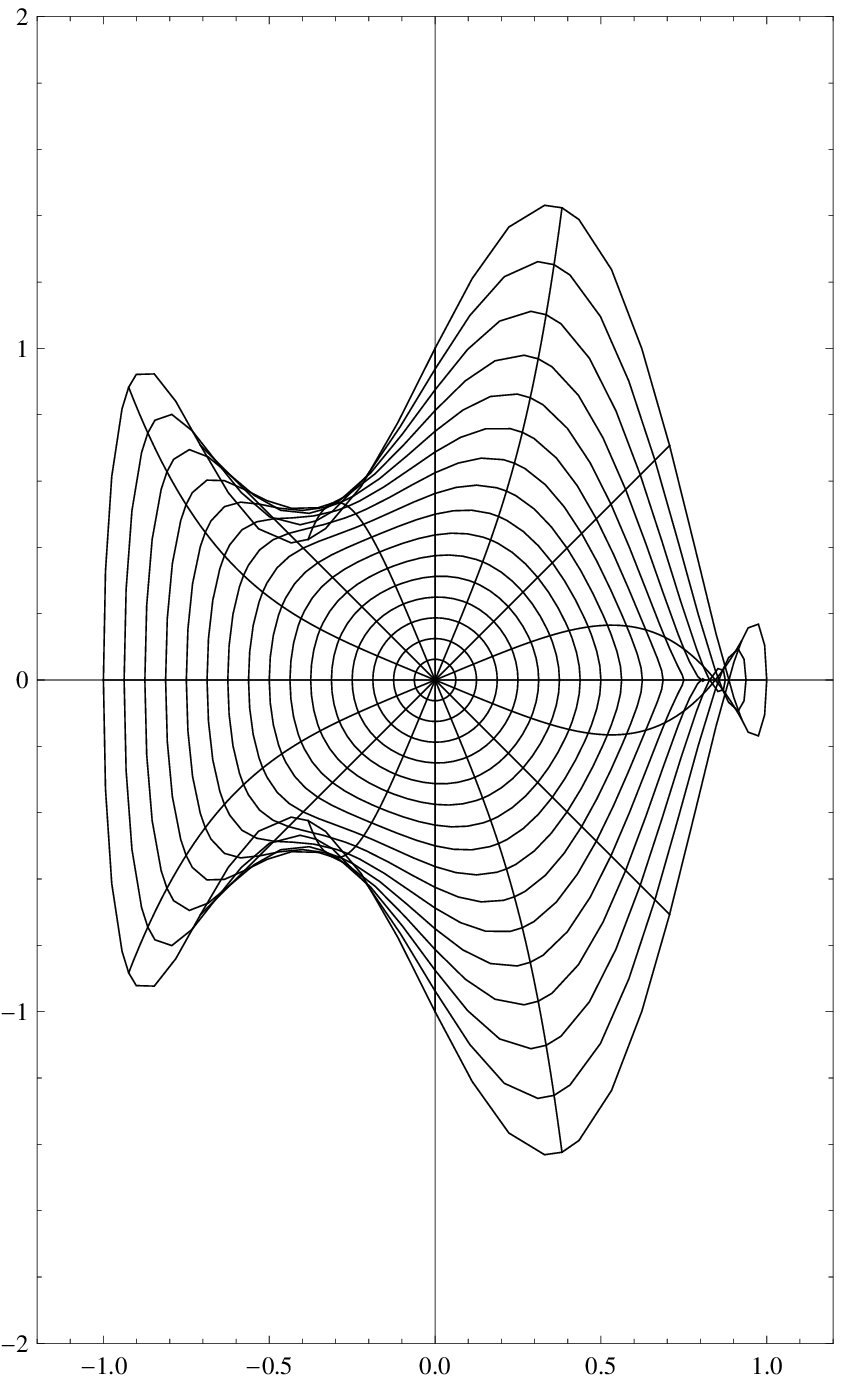}}\hspace{5pt}
  \subfigure[$p_2*q_2$]{\includegraphics[width=1.5in,height=2in]{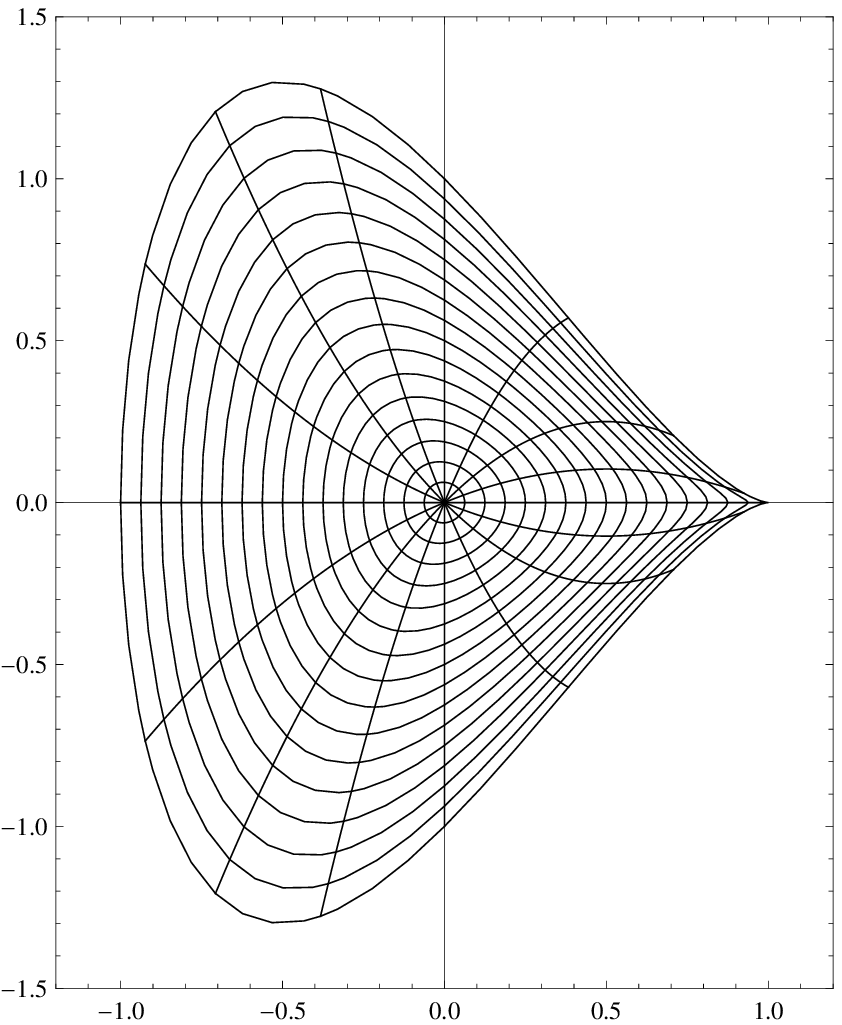}}\hspace{5pt}
  \subfigure[$p_3*q_3$]{\includegraphics[width=1.5in,height=2in]{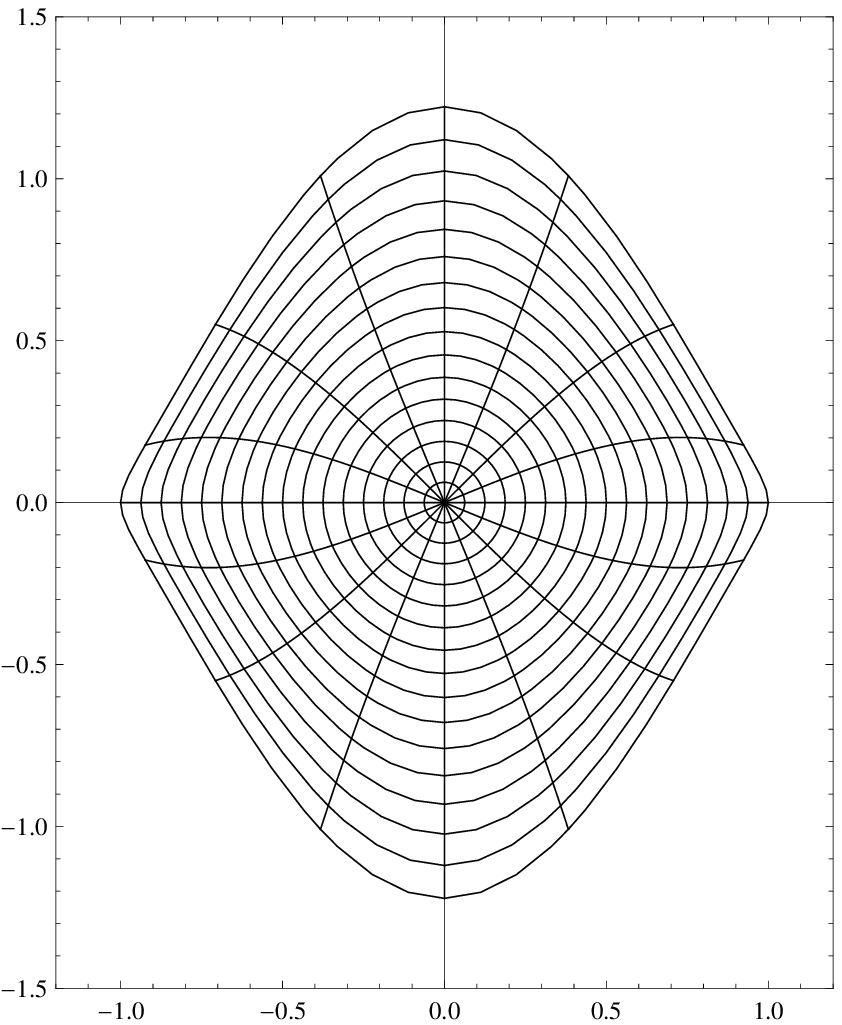}}\hspace{5pt}
  \subfigure[$p_4*q_4$]{\includegraphics[width=1.5in,height=2in]{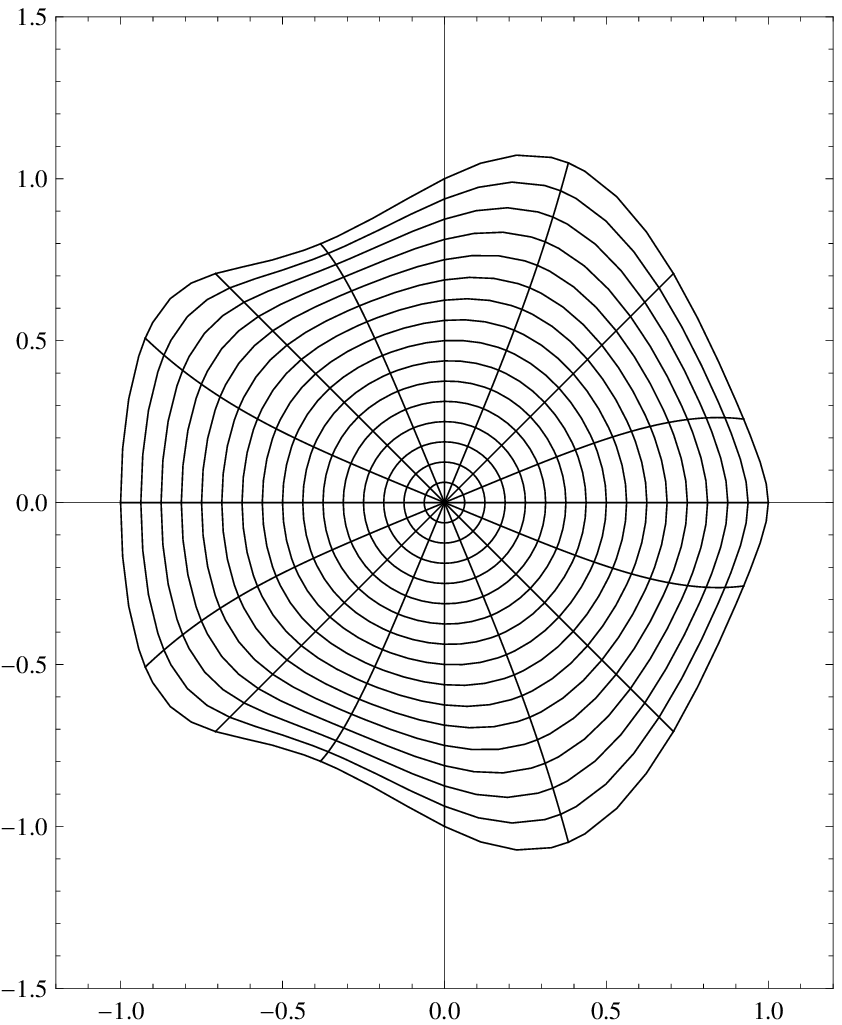}}
  \caption{Images of the functions $q_n(z)=z-z^n/n+\bar{z}^n/n$ and the convolutions $p_n*q_n$ for $n=2,3,4$.}\label{fig7}
\end{center}
\end{figure}

It is easy to see that $\RE (u_n*r_n)'>1/2$ in $\mathbb{D}$ so that the convolutions $p_n*q_n$ $(n=2,3,\ldots)$ are univalent and convex in the direction of imaginary axis, by Corollary \ref{cor3.2}. The images of the unit disk under $q_n$ and $p_n*q_n$ are depicted in Figure \ref{fig7} for $n=2,3,4$.
\end{example}

\begin{example}\label{ex3.4}
For $k=1,2,\ldots$, let $\Psi_k=\gamma_k+\overline{\delta}_k$ be the shears of the identity map in the direction of imaginary axis with dilatation $w_{\Psi_k}(z)=z^k$. Then $\Psi_k \in \mathcal{W}_H^+(z)$ and
\[\gamma_k(z)=z+\sum_{n=1}^{\infty}(-1)^n \frac{z^{nk+1}}{nk+1},\quad \delta_k(z)=\sum_{n=1}^{\infty}(-1)^{n+1}\frac{z^{nk+1}}{nk+1} \quad (z \in \mathbb{D};k=1,2,\ldots).\]
In particular, $\Psi_1(z)=\overline{z}+2i \arg \{1+z\}$ and $\Psi_2(z)=\overline{z}+2i \IM(\arctan z)$.

\begin{figure}[here]
\begin{center}
  \subfigure[$\Psi_1$]{\includegraphics[width=1.5in,height=2in]{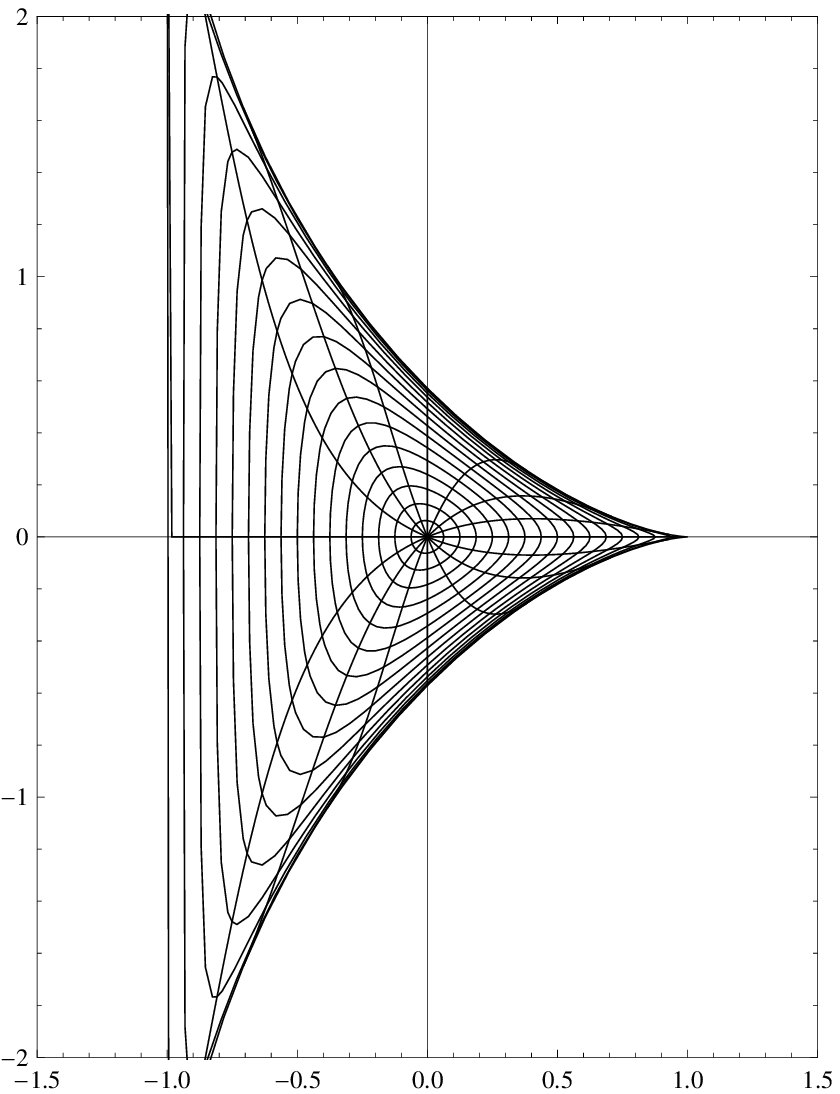}}\hspace{5pt}
  \subfigure[$\Psi_2$]{\includegraphics[width=1.5in,height=2in]{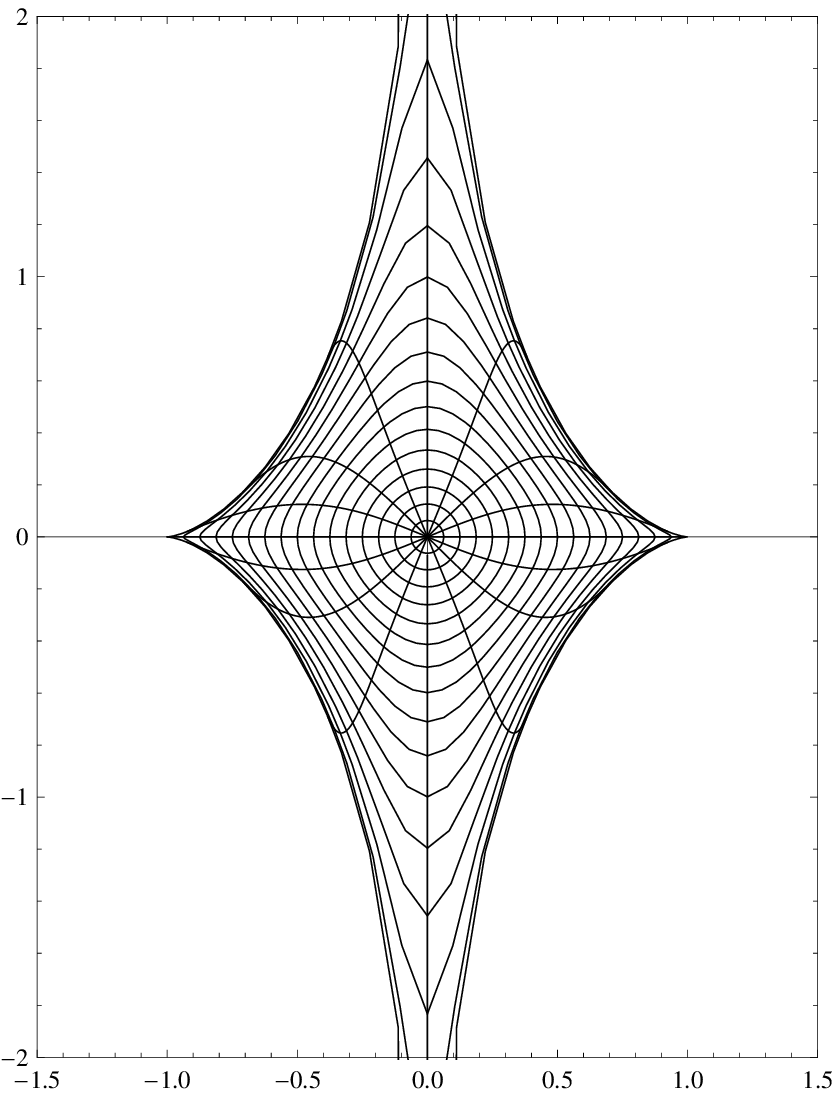}}\hspace{5pt}
  \subfigure[$\Psi_3$]{\includegraphics[width=1.5in,height=2in]{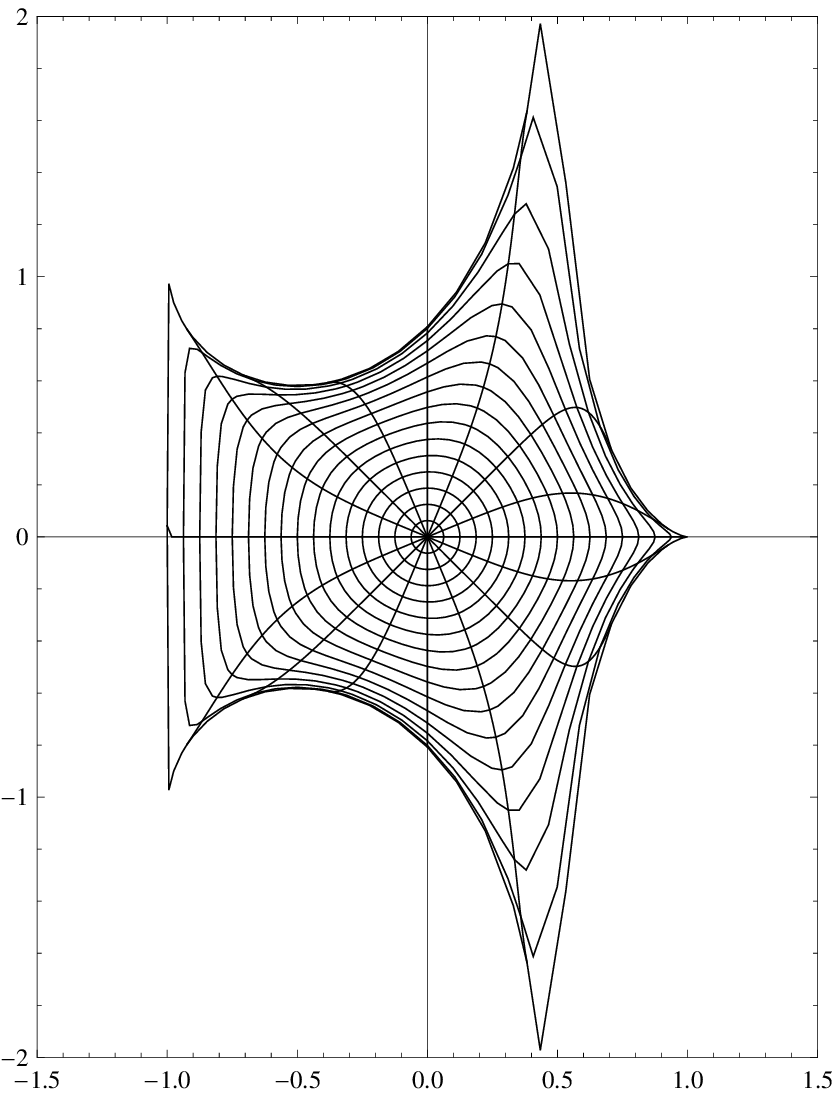}}\hspace{5pt}
  \subfigure[$\Gamma_1*\Psi_1$]{\includegraphics[width=1.5in,height=2in]{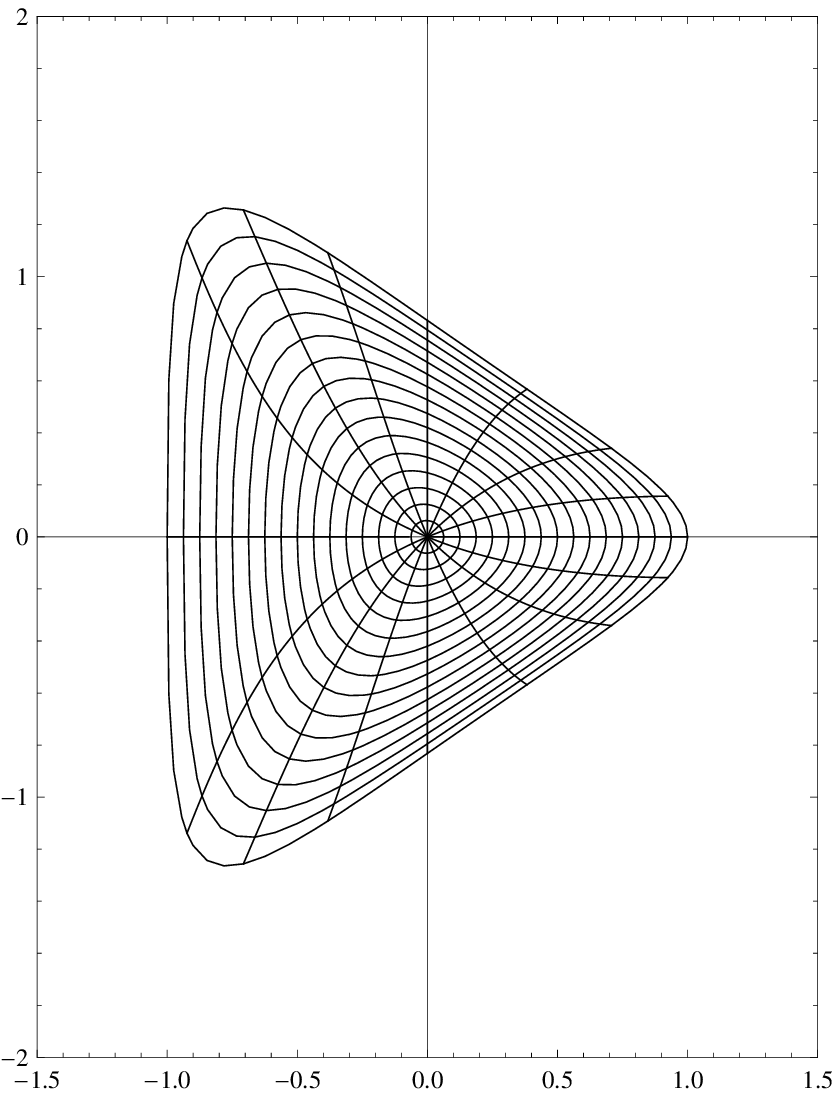}}\hspace{5pt}
  \subfigure[$\Gamma_2*\Psi_2$]{\includegraphics[width=1.5in,height=2in]{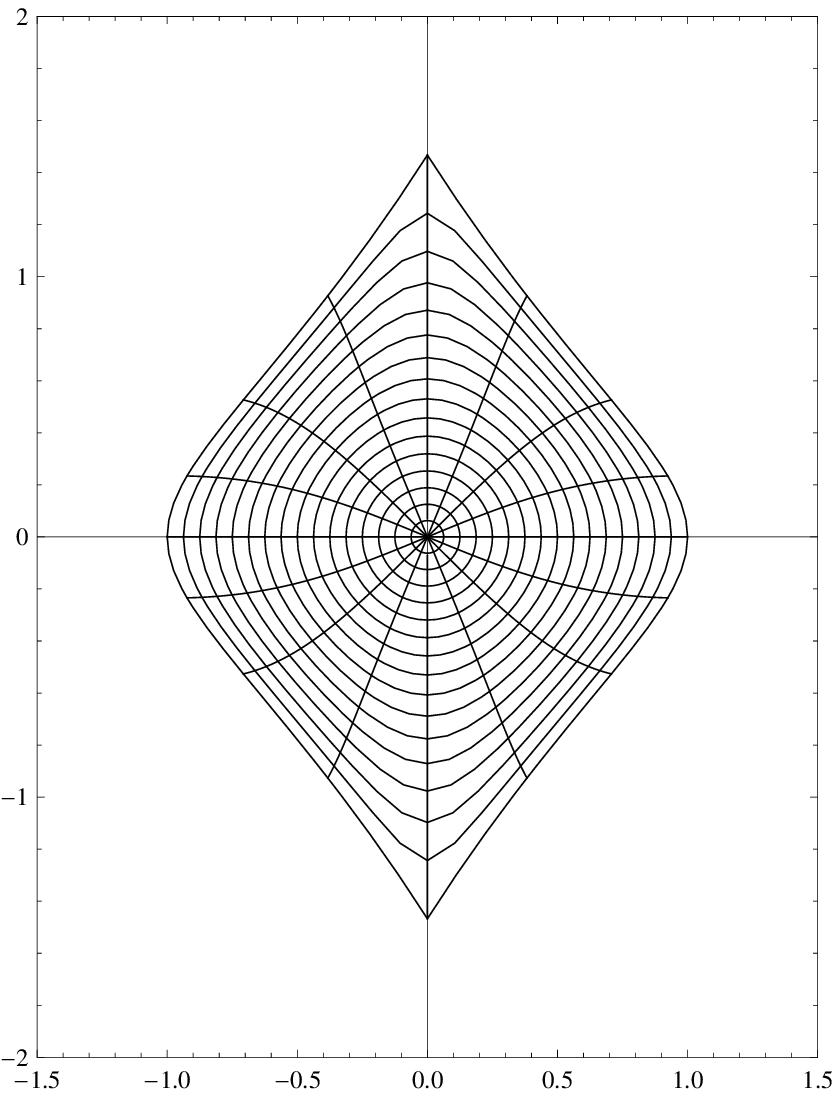}}\hspace{5pt}
  \subfigure[$\Gamma_3*\Psi_3$]{\includegraphics[width=1.5in,height=2in]{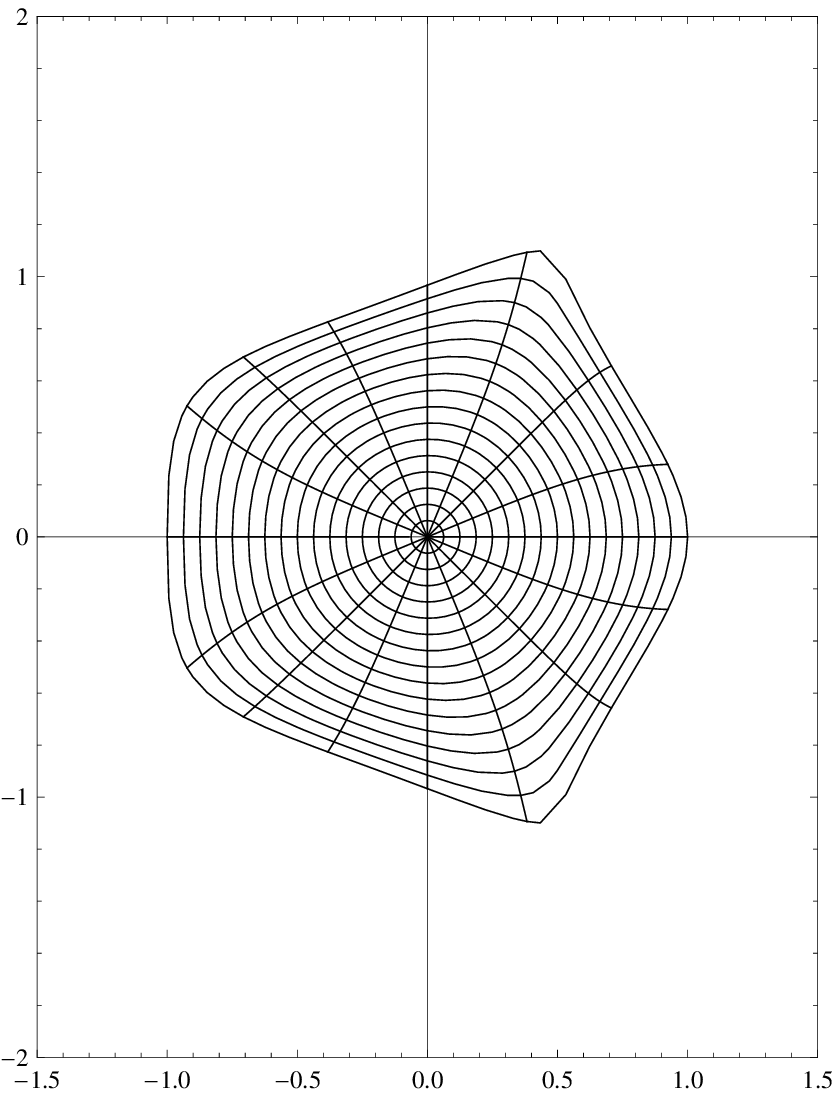}}
  \caption{Images of the functions $\Psi_k$ and the convolutions $\Gamma_k*\Psi_k$ for $k=1,2,3$.}\label{fig8}
\end{center}
\end{figure}

Considering the convolutions $\Gamma_k *\Psi_k$ $(k=1,2,\ldots)$ where $\Gamma_k=\mu_k+\overline{\nu}_k$ are defined in Example \ref{ex2.5}, we see that $\RE (z(\mu_k*\gamma_k)''(z)+(\mu_k*\gamma_k)'(z))=\RE 1/(1+z^k)>1/2$, for $z \in \mathbb{D}$, which imply that $\RE (\mu_k*\gamma_k)'>\log 2$ by \cite[Theorem 2]{silverman}. By Corollary \ref{cor3.2}, it follows that the convolutions
\[(\Gamma_k*\Psi_k)(z)=z+\sum_{n=1}^{\infty}(-1)^n\frac{z^{nk+1}}{(nk+1)^2}+\overline{\sum_{n=1}^{\infty}(-1)^{n+1}\frac{z^{nk+1}}{(nk+1)^2}} \quad (z \in \mathbb{D},k=1,2,\ldots)\]
are univalent and convex in the direction of imaginary axis (see Figure 8).
\end{example}

On taking $f_1 \equiv \Gamma_1$ in Theorem \ref{th3.1} where $\Gamma_1 \in \mathcal{W}_{H}^{-}(z)$ is given by \eqref{eq2.1}, we obtain the following corollary. Its proof being similar to the proof of Theorem \ref{th2.6} is omitted.

\begin{corollary}\label{cor3.5}
Let $f=h+\overline{g} \in \mathcal{W}_{H}^{+}(\phi)$ with $\RE h(z)/\phi(z)>1/2$ for all $z \in \mathbb{D}$. If the analytic function $\mu_1*\phi$ is univalent and convex in the direction of imaginary axis, then $\Lambda_{H}^+[f] \in \mathcal{S}_{H}^{0}$ and is convex in the direction of imaginary axis, $\Lambda_{H}^+$ being the positive harmonic Alexander operator defined by \eqref{eq2.2}.
\end{corollary}

Next, we provide some examples that illustrate Corollary \ref{cor3.5} for specific choices of $\phi$.

\begin{example}\label{ex3.6}
Consider a non-univalent harmonic function $f=h+\overline{g}$ where $h(z)=z/(1-z^2)^2$ and $g(z)=z^3/(1-z^2)^2$, belonging to the class $\mathcal{W}_H^+(\phi)$, $\phi(z)=z(1+z^2)/(1-z^2)^2$. Note that $\RE h(z)/\phi(z)=\RE 1/(1+z^2)>1/2$ and $(\mu_1*\phi)(z)=z/(1-z^2)$ is univalent and convex in the direction of imaginary  axis. By Corollary \ref{cor3.5}, the convolution
\[(\Gamma_1*f)(z)=\RE \frac{z}{1-z^2}+\frac{1}{2}i \arg \left(\frac{1+z}{1-z}\right),\quad z\in \mathbb{D}\]
is univalent and convex in the direction of imaginary axis (see Figure \ref{fig9}).
\begin{figure}[here]
\begin{center}
  \subfigure{\includegraphics[width=2.8in,height=1.8in]{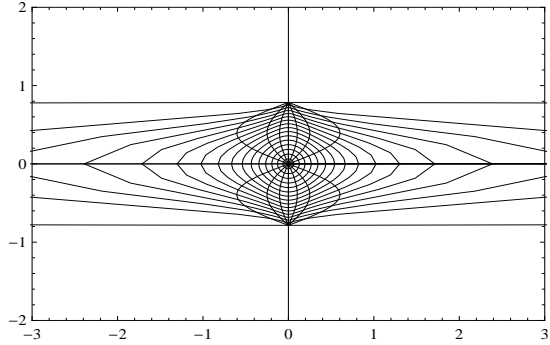}}
  \caption{Image of the convolution of $\Gamma_1$ with $f(z)=z/(1-z^2)^2+\overline{z}^3/(1-\overline{z}^2)^2$.}\label{fig9}
\end{center}
\end{figure}
\end{example}

\begin{example}
The harmonic half-plane mapping $L=U-\overline{V}$, $U$ and $V$ are defined in Example \ref{ex2.12}, belongs to $\mathcal{W}_{H}^+(z/(1-z))$ and satisfies the hypothesis of Corollary \ref{cor3.5}. Therefore the mapping
\[\Lambda_{H}^+[L](z)=-\log |1-z|+i \IM \frac{z}{1-z},\quad z\in \mathbb{D}\]
is univalent and convex in the direction of imaginary axis (see Figure \ref{fig10}).

\begin{figure}[here]
\begin{center}
  \subfigure{\includegraphics[width=2.8in,height=3in]{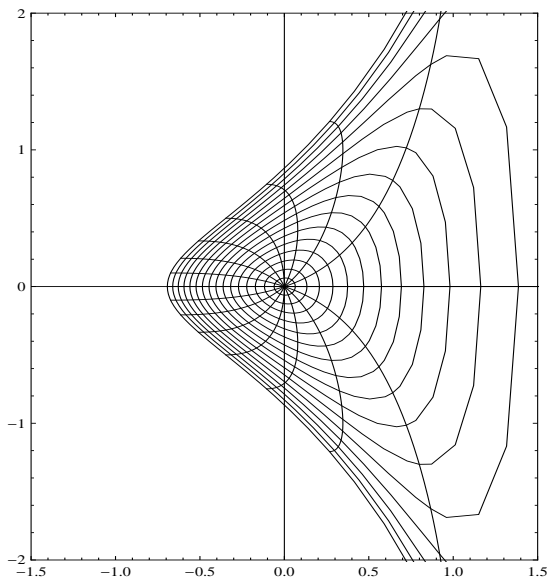}}
  \caption{Image of the convolution of $\Gamma_1$ with the harmonic half-plane mapping $L$.}\label{fig10}
\end{center}
\end{figure}
\end{example}

\section*{Acknowledgements}
The research work presented here is supported by research fellowship from Council of Scientific and
Industrial Research (CSIR), New Delhi and a grant from University of Delhi, Delhi.


\begin{thebibliography}{12}

\bibitem{cluniesheilsmall} J. Clunie\ and\ T. Sheil-Small, \emph{Harmonic univalent functions}, Ann.\  Acad.\  Sci.\  Fenn.\  Ser.\  A I Math.\  {\bf 9} (1984), 3--25.

\bibitem{dorff1} M. Dorff, \emph{Convolutions of planar harmonic convex mappings}, Complex Var.\  Theory Appl.\  {\bf 45} (2001), no.~3, 263--271.

\bibitem{dorff2} M. Dorff, M. Nowak\ and\ M. Wo\l oszkiewicz, \emph{Convolutions of harmonic convex mappings}, Complex Var.\  Elliptic Equ.\  {\bf 57} (2012), no.~5, 489--503.

\bibitem{goodloe} M. R. Goodloe, \emph{Hadamard products of convex harmonic mappings}, Complex Var.\  Theory Appl.\  {\bf 47} (2002), no.~2, 81--92.

\bibitem{monograph} S. S. Miller\ and\ P. T. Mocanu, {\it Differential subordinations}, Monographs and Textbooks in Pure and Applied Mathematics, 225, Dekker, New York, 2000.

\bibitem{mocanu} P. T. Mocanu, \emph{Injectivity conditions in the complex plane}, Complex Anal. Oper. Theory {\bf 5} (2011), no.~3, 759--766.

\bibitem{sumit1} S. Nagpal and V. Ravichandran, \emph{A subclass of close-to-convex harmonic mappings}, Complex Var.\  Elliptic Equ.\  (2012), \texttt{DOI:
    10.1080/17476933.2012.727409}.

\bibitem{sumit2} S. Nagpal and V. Ravichandran, \emph{Fully starlike and fully convex harmonic mappings of order $\alpha$}, Annales Polonici Mathematici, \textbf{108} (2013), 85--107.

\bibitem{sumit3} S. Nagpal and V. Ravichandran, \emph{Construction of subclasses of univalent harmonic mappings}, arXiv:\texttt{1209.0075}.

\bibitem{sumit4} S. Nagpal and V. Ravichandran, \emph{On a subclass of close-to-convex harmonic mappings}, arXiv:\texttt{1207.3404}.

\bibitem{sumit5} S. Nagpal and V. Ravichandran, \emph{A class of harmonic functions with real coefficients defined by convolution}, arXiv:\texttt{1301.2746}.

\bibitem{ruscheweyh} S. Ruscheweyh\ and\ L. C. Salinas, \emph{On the preservation of direction-convexity and the Goodman-Saff conjecture}, Ann.\  Acad.\  Sci.\  Fenn.\  Ser.\  A I Math.\  {\bf 14} (1989), no.~1, 63--73.

\bibitem{silverman} H. Silverman, \emph{A class of bounded starlike functions}, Internat.\  J. Math.\  Math.\  Sci.\  {\bf 17} (1994), no.~2, 249--252.



\end{thebibliography}
\end{document}